\documentclass[15pt, a4paper]{article}
\usepackage[latin1]{inputenc}
\usepackage[english]{babel}
\usepackage[margin=1in]{geometry}

\usepackage{blindtext}
\usepackage{amssymb}
\usepackage{hyperref}
\hypersetup{
    colorlinks=true,
    linkcolor=blue,
    filecolor=magenta,      
    urlcolor=cyan,
}
 
\usepackage{amsthm}
\usepackage[normalem]{ulem}
\usepackage{amsmath}
\usepackage{booktabs}
\usepackage{color}
\usepackage{array}
\usepackage{cite}
\usepackage{soul}
\usepackage{mathtools}
\usepackage{bm}
\usepackage{multirow}
\usepackage{makecell}
\usepackage{thmtools}
\usepackage{thm-restate}
\newtheorem{theorem}{Theorem}[section]
\newtheorem{corollary}[theorem]{Corollary}
\newtheorem{lemma}[theorem]{Lemma}
\newtheorem{proposition}[theorem]{Proposition}

\newtheorem{example}{Example}
\newtheorem{remark}[theorem]{Remark}

\usepackage{cleveref}

\newcommand\nc\newcommand
\nc{\cA}{\mathcal{A}}\nc{\cB}{\mathcal{B}}\nc{\cC}{\mathcal{C}}\nc{\cD}{\mathcal{D}}
\nc{\cE}{\mathcal{E}}\nc{\cF}{\mathcal{F}}\nc{\cG}{\mathcal{G}}\nc{\cH}{\mathcal{H}}
\nc{\cI}{\mathcal{I}}\nc{\cJ}{\mathcal{J}}\nc{\cK}{\mathcal{K}}\nc{\cL}{\mathcal{L}}
\nc{\cM}{\mathcal{M}}\nc{\cN}{\mathcal{N}}\nc{\cO}{\mathcal{O}}\nc{\cP}{\mathcal{P}}
\nc{\cQ}{\mathcal{Q}}\nc{\cR}{\mathcal{R}}\nc{\cS}{\mathcal{S}}\nc{\cT}{\mathcal{T}}
\nc{\cU}{\mathcal{U}}\nc{\cV}{\mathcal{V}}\nc{\cW}{\mathcal{W}}\nc{\cX}{\mathcal{X}}
\nc{\cY}{\mathcal{Y}}\nc{\cZ}{\mathcal{Z}}

\nc{\bba}{\mathbf{a}}\nc{\bbb}{\mathbf{b}}\nc{\bbc}{\mathbf{c}}\nc{\bbd}{\mathbf{d}}
\nc{\bbe}{\mathbf{e}}\nc{\bbf}{\mathbf{f}}\nc{\bbg}{\mathbf{g}}\nc{\bbh}{\mathbf{h}}
\nc{\bbi}{\mathbf{i}}\nc{\bbj}{\mathbf{j}}\nc{\bbk}{\mathbf{k}}\nc{\bbl}{\mathbf{l}}
\nc{\bbm}{\mathbf{m}}\nc{\bbn}{\mathbf{n}}\nc{\bbo}{\mathbf{o}}\nc{\bbp}{\mathbf{p}}
\nc{\bbq}{\mathbf{q}}\nc{\bbr}{\mathbf{r}}\nc{\bbs}{\mathbf{s}}\nc{\bbt}{\mathbf{t}}
\nc{\bbu}{\mathbf{u}}\nc{\bbv}{\mathbf{v}}\nc{\bbw}{\mathbf{w}}\nc{\bbx}{\mathbf{x}}
\nc{\bby}{\mathbf{y}}\nc{\bbz}{\mathbf{z}}

\nc{\bbA}{\mathbf{A}}\nc{\bbB}{\mathbf{B}}\nc{\bbC}{\mathbf{C}}\nc{\bbD}{\mathbf{D}}
\nc{\bbE}{\mathbf{E}}\nc{\bbF}{\mathbf{F}}\nc{\bbG}{\mathbf{G}}\nc{\bbH}{\mathbf{H}}
\nc{\bbI}{\mathbf{I}}\nc{\bbJ}{\mathbf{J}}\nc{\bbK}{\mathbf{K}}\nc{\bbL}{\mathbf{L}}
\nc{\bbM}{\mathbf{M}}\nc{\bbN}{\mathbf{N}}\nc{\bbO}{\mathbf{O}}\nc{\bbP}{\mathbf{P}}
\nc{\bbQ}{\mathbf{Q}}\nc{\bbR}{\mathbf{R}}\nc{\bbS}{\mathbf{S}}\nc{\bbT}{\mathbf{T}}
\nc{\bbU}{\mathbf{U}}\nc{\bbV}{\mathbf{V}}\nc{\bbW}{\mathbf{W}}\nc{\bfX}{\mathbf{X}}
\nc{\bbY}{\mathbf{Y}}\nc{\bbZ}{\mathbf{Z}}

\nc{\sA}{\mathsf{A}}\nc{\sB}{\mathsf{B}}\nc{\sC}{\mathsf{C}}\nc{\sD}{\mathsf{D}}
\nc{\sE}{\mathsf{E}}\nc{\sF}{\mathsf{F}}\nc{\sG}{\mathsf{G}}\nc{\sH}{\mathsf{H}}
\nc{\sI}{\mathsf{I}}\nc{\sJ}{\mathsf{J}}\nc{\sK}{\mathsf{K}}\nc{\sL}{\mathsf{L}}
\nc{\sM}{\mathsf{M}}\nc{\sN}{\mathsf{N}}\nc{\sO}{\mathsf{O}}\nc{\sP}{\mathsf{P}}
\nc{\sQ}{\mathsf{Q}}\nc{\sR}{\mathsf{R}}\nc{\sS}{\mathsf{S}}\nc{\sT}{\mathsf{T}}
\nc{\sU}{\mathsf{U}}\nc{\sV}{\mathsf{V}}\nc{\sW}{\mathsf{W}}\nc{\sX}{\mathsf{X}}
\nc{\sY}{\mathsf{Y}}\nc{\sZ}{\mathsf{Z}}

\newcommand{\abs}[1]{\left|#1\right|}

\newcommand{\parenv}[1]{\left( #1 \right)}

\nc{\set}[1]{\llbracket #1 \rrbracket}

\newcommand{\bal}[1]{\begin{align}\label{#1}}
\newcommand{\eal}{\end{align}}
\renewcommand{\leq}{\leqslant}

\renewcommand{\geq}{\geqslant}

\renewcommand{\Bbb}{\mathbb}

\renewcommand{\Bbb}{\mathbb}

\newcommand{\Fq}{{{\Bbb F}}_{\!q}}

\theoremstyle{definition}
\newtheorem{definition}{Definition}[section]

\newcommand{\qbinom}[2]{\genfrac{[}{]}{0pt}{}{#1}{#2}}

\title{A point-variety incidence theorem over finite fields, and its applications}
\author{Xiangliang~Kong and Itzhak~Tamo%
\thanks{X. Kong (rongxlkong@gmail.com) and I. Tamo (tamo@tauex.tau.ac.il) are with the Department of Electrical Engineering-Systems, Tel Aviv University, Tel Aviv-Yafo 6997801, Israel. This work was supported by the European Research Council (ERC) under Grant 852953.}
}
\date{}

\begin{document}

\maketitle

\begin{abstract}
Incidence problems between geometric objects is a key area of focus in the field of discrete geometry. Among them, the study of incidence problems over finite fields have received a considerable amount of attention in recent years. 

In this paper, by characterizing the singular values and singular vectors of the corresponding incidence matrix through group algebras, we prove a bound on the number of incidences between points and varieties of a certain form over finite fields. Our result leads to a new incidence bound for points and flats in finite geometries, which improves previous results for certain parameter regimes. As another application of our point-variety incidence bound, we extend a result on pinned distance problems by Phuong, Thang, and Vinh, and independently by Cilleruelo, Iosevich, Lund, Roche-Newton, and Rudnev, under a weaker condition.
\end{abstract}

\section{Introduction}

As one of the central topics in discrete geometry, the structure of incidences between points and various geometric objects have been extensively studied over the decades. As the starting point of this series of studies, in 1983, Szemer\'{e}di and Trotter \cite{ST83} proved that the number of incidences $I(\cP,\cL)$ between a set of points $\cP$ and a set of lines $\cL$ in the $2$-dimensional real plane $\mathbb{R}^2$ is bounded by $O(|\cP|^{\frac{2}{3}}|\cL|^{\frac{2}{3}}+|P|+|L|)$. Apart from being an interesting result in itself, the Szemer\'{e}di-Trotter theorem has been applied to numerous problems and has various extensions and generalizations (see, for examples, \cite{TV06,Dvir12,Sheffer22,Kharazishvili24}).

Let $q$ be a prime power and denote $\Fq$ as the finite field with $q$ elements. Analogous results to the Szemer\'{e}di-Trotter theorem for point-line incidences in $\Fq^2$ were first considered by Bourgain, Katz, and Tao \cite{BKT04} due to their close connection to sum-product estimates. Following this work, incidence problems over finite fields have been extensively studied (see, for examples, \cite{IK09, Vinh11,HR11,Jones12,Grosu14,Kollar15,Lewko15,Lewko19,KPV21,KLP22}). 

As one of the point-line incidence bounds over finite fields, Vinh \cite{Vinh11} proved that $\abs{I(\cP,\cL)-|\cP||\cL|/q}\leq \sqrt{q|\cP||\cL|}$ for a point set $\cP$ and a line set $\cL$ in $\Fq^2$. This bound is shown to be tight when $|\cP|=|\cL|=q^{3/2}$. By observing that $I(\cP, \cL)$ equals the number of edges in the point-line incidence graph over $\Fq^2$ induced by $\cP$ and $\cL$, Vinh \cite{Vinh11} proved his result using the famous expander mixing lemma by Alon and Chung \cite{AC88}.

Inspired by Vinh's work \cite{Vinh11}, using spectral graph theory to bound incidences between points and various geometric objects over finite fields has become a standard method. For examples, by generalizing Vinh's argument in \cite{Vinh11}, Lund and Saraf \cite{LS14} obtained an incidence bound for points and balanced incomplete block designs (BIBDs), Phuong, Thang, and Vinh \cite{PTV17} obtained an incidence bound for points and generalized spheres, Pham, Phuong, Sang, Valculescu, and Vinh \cite{PPSSVV18} proved lower bounds on the number of distinct distances between points and lines in the plane over finite fields. The main tool from spectral graph theory used in most of these works is the following bipartite-variant of the well-known expander mixing lemma by Alon and Chung \cite{AC88}.

\begin{lemma}\label{expander_mix_bipartite}(Expander crossing lemma, \cite{Haemers95,DSV12})
Let $G$ be a bipartite graph with parts $A,B$ such that the vertices in $A$ all have degree $a$ and the vertices in $B$ all have degree $b$. Then, for any $X\subseteq A$ and $Y\subseteq B$, the number of edges between $X$ and $Y$, denoted by $e(X,Y)$, satisfies
\begin{equation}\label{eq_expander_cross}
    \left|e(X,Y)-\frac{a}{|B|}|X||Y|\right|\leq \lambda_3\sqrt{|X||Y|\parenv{1-\frac{|X|}{|A|}}\parenv{1-\frac{|Y|}{|B|}}},
\end{equation}
where $\lambda_3$ is the third eigenvalue of $G$\footnote{We denote $|\lambda_1| \geq |\lambda_2| \geq |\lambda_3|\geq \cdots$ as the eigenvalues of the adjacency matrix of $G$, ordered by their absolute value.}.
\end{lemma}

In this paper, we study an incidence problem between points and a certain family of varieties over finite fields, which was first considered in \cite{PPSVV16}. Similar to a previous work \cite{Zachi23} of the second author, our approach employs a singular value decomposition of the incidence matrix of the point-variety incidence graph, coupled with an analysis of the related group algebras. 

The main idea of our approach is also from spectral graph theory. However, unlike previous works, we don't directly apply Lemma \ref{expander_mix_bipartite} to the point-variety incidence graph. To explain the difference, we first briefly recall the proof of Lemma \ref{expander_mix_bipartite}. 
Let $\bbM$ be the adjacency matrix of the bipartite graph $G$. It is straightforward to verify that the number of edges between subsets $X \subseteq A$ and $Y \subseteq B$, denoted by $e(X, Y)$, satisfies
$$
e(X, Y) = \mathbf{1}_X \cdot \bbM \cdot \mathbf{1}_Y',
$$
where $\mathbf{1}_X$ and $\mathbf{1}_Y$ are the characteristic vectors of $X$ and $Y$, respectively.

We can decompose $\mathbf{1}_X$ and $\mathbf{1}_Y$ as follows:
$$
\mathbf{1}_X = \widetilde{\mathbf{1}_X} + \overline{\mathbf{1}_X}, \quad \mathbf{1}_Y = \widetilde{\mathbf{1}_Y} + \overline{\mathbf{1}_Y},
$$
where $\widetilde{\mathbf{1}_X}$ and $\widetilde{\mathbf{1}_Y}$ are the projections onto the span of the eigenspaces corresponding to the first and second eigenvalues $\lambda_1$ and $\lambda_2$, and $\overline{\mathbf{1}_X}$ and $\overline{\mathbf{1}_Y}$ are their orthogonal complements. Substituting these decompositions into the expression for $e(X, Y)$, we obtain
$$
e(X, Y) = \widetilde{\mathbf{1}_X} \cdot \bbM \cdot \widetilde{\mathbf{1}_Y}' + \overline{\mathbf{1}_X} \cdot \bbM \cdot \overline{\mathbf{1}_Y}',
$$
where the cross terms vanish due to the orthogonality of the eigenspaces. The first term evaluates to $\frac{a}{|B|} |X| |Y|$, and the core of the proof lies in bounding the absolute value of the second term. A common approach is to assume that both $\overline{\mathbf{1}_X}$ and $\overline{\mathbf{1}_Y}$ lie entirely within the eigenspace associated with the third eigenvalue, from which the bound in Lemma \ref{expander_mix_bipartite} follows.

Our proof follows the same overall strategy but refines the bound on the second term using the following key observation: When the bipartite graph $G$ is highly unbalanced, such as when $|A|$ is significantly larger than $|B|$, the multiplicity of the zero eigenvalue of $M$ becomes very large. In other words, the zero eigenspace has a high dimension, spanning nearly the entire space. As a result, $\overline{\mathbf{1}_X}$ and $\overline{\mathbf{1}_Y}$ are highly likely to have substantial projections onto the zero eigenspace. Since their components in the zero eigenspace do not contribute to the second term, we obtain a sharper estimate, leading to an improved bound on $e(X,Y)$.

As we will introduce in detail later in Section \ref{subsec: problem setting}, the point-variety incidence graph considered in this work is indeed highly unbalanced: the number of varieties is significantly larger than the number of points. Therefore, by precisely characterizing the zero eigenspace, we confirm that  $\overline{\mathbf{1}_X}$ and $\overline{\mathbf{1}_Y}$ have a large projection onto the zero eigenspace, which facilitates the improved bound.

\subsection{Problem settings and main results}\label{subsec: problem setting}

Let $S$ be a set of polynomials in $\Fq[x_1,\ldots,x_n]$. The variety determined by $S$ is defined as 
$$V(S)\triangleq\{\bbp\in \Fq^{n}:~f(\bbp)=0~\mathrm{for~all}~f\in S\}.$$

Before introducing the formal definition of the general varieties considered in this paper, we first provide motivation by examining two fundamental examples. These examples serve as special cases of the broader class of varieties discussed later and play a crucial role in deriving bounds on point-flat incidences (see Theorem \ref{main_thm2}) and the pinned-distance problem (see Corollary \ref{pinned_distance_improved}).

\begin{example}\label{ex1}
Let $ \bba_i = (a_{i,1},\ldots,a_{i,n+1}) \in \mathbb{F}_q^{n+1} $ for $ 1\leq i\leq d $ be $ d $ vectors, and consider the following set of $ d $ linear polynomials in $ \mathbb{F}_q[x_1, \ldots, x_{n+d}] $:
\begin{align}
    S = \left\{\sum_{j=1}^{n} a_{i,j} x_j + a_{i,n+1} - x_{n+i} \right\}_{1\leq i\leq d}. \label{eq_flat_0}
\end{align}
It is straightforward to verify that the variety defined by these polynomials corresponds to an $ n $-flat in $ \mathbb{F}_q^{n+d} $. We denote this variety as $ F_{\bba_1,\ldots,\bba_n} $. 

For instance, when $ d=1 $, the polynomial set simplifies to:
\begin{align*}
    S = \left\{\sum_{j=1}^{n} a_{1,j} x_j + a_{1,n+1} - x_{n+1} \right\},
\end{align*}
and the variety $ F_{\bba_1} = V(S) $ represents a hyperplane in $ \mathbb{F}_q^{n+1} $. In this case, the family of hyperplanes $ \{F_{\bba_1} : \bba_1 \in \mathbb{F}_q^{n+1}\} $ consists of all hyperplanes in $ \mathbb{F}_q^{n+1} $ except those with normal vector $ (0,\ldots,0,1) $.
\end{example}

\begin{example}\label{ex2}
Given a point $ \bbp = (p_1,\ldots,p_n) \in \mathbb{F}_q^n $, consider the variety defined by the following quadratic polynomial in $ \mathbb{F}_q[x_1, \ldots, x_{n+1}] $:
\begin{align}
    \sum_{i=1}^n (x_i - p_i)^2 - x_{n+1}. \label{eq_poly_p}
\end{align}
This variety represents an $ n $-dimensional elliptical paraboloid in $ \mathbb{F}_q^{n+1} $ with its vertex at $ (\bbp,0) $.
\end{example}

Now, we present the formal definition of the general variety we consider in this work. Let $h_1,\ldots,h_d\in \Fq[x_1,\ldots,x_n]$ be fixed polynomials of total degree at most $q-1$. Let $\bbb_{i}=(b_{i,1},\ldots,b_{i,n})$, $1\leq i\leq d$, be $d$ fixed vectors in $(\mathbb{Z}^{+})^n$ such that $\mathrm{gcd}(b_{i,j},q-1)=1$ for all $1\leq j\leq n$. Note that for a positive integer $b$, when $\mathrm{gcd}(b, q-1) = 1$, the function $x \mapsto x^b$ is a permutation of the field. For any $d$-tuple $(\bba_1,\ldots,\bba_{d})$ with $\bba_{i}=(a_{i,1},\ldots,a_{i,n+1})\in \Fq^{n+1}$, 
we define polynomials $f_{\bba_{1}},\ldots,f_{\bba_{d}}\in \Fq[x_1,\ldots,x_n]$ as
\begin{equation}\label{eq_variety_0}
f_{\bba_{i}}(x_1,\ldots,x_n)\triangleq \sum_{j=1}^{n}a_{i,j}x_{j}^{b_{i,j}}+a_{i,n+1}.
\end{equation}
Then, the variety $V_{\bba_1,\ldots,\bba_{d}}\subseteq \Fq^{n+d}$ defined by $(\bba_1,\ldots,\bba_d)$ is given by:
\begin{align}
    V_{\bba_1,\ldots,\bba_{d}}&\triangleq V(x_{n+1}-(h_1+f_{\bba_1})(x_1,\ldots,x_n),\ldots,x_{n+d}-(h_d+f_{\bba_d})(x_1,\ldots,x_n)).\label{eq_variety_1}
\end{align}
Note that the variety $V_{\bba_1,\ldots,\bba_{d}}$ also depends on the vectors $\bbb_{1},\ldots,\bbb_d$. However, since these vectors are fixed, we omit them from the notation of the variety for simplicity.

Notice  that the above two examples are special cases of the variety defined in \eqref{eq_variety_1}. Indeed, when we take $ \bbb_i = (1,\ldots,1) $ and $ h_i(\bbx) \equiv 0 $ for all $ 1\leq i\leq d $ in (\ref{eq_variety_0}) and (\ref{eq_variety_1}), the variety $ V_{\bba_1,\ldots,\bba_{d}} $ becomes  
\begin{align}
    \left\{(x_1,\ldots,x_{n+d})\in \mathbb{F}_q^{n+d} : x_{n+i}=\sum_{j=1}^{n}a_{i,j}x_{j}+a_{i,n+1},\ \forall~1\leq i\leq d\right\}, \label{eq_flat_1}
\end{align}
which is the $ n $-flat $ F_{\bba_1,\ldots,\bba_d} $ introduced in Example \ref{ex1}. Similarly, for $d=1$, $h(\bbx)=\sum_{i=1}^{n}x_i^2$, $\bbb_1=(1,\ldots,1)$
and $\bba_1=(-2p_1,\ldots,-2p_n,\sum_{i=1}^{n}p_i^2)$  for some fixed $\bbp=(p_1,\ldots,p_n)\in \Fq^n$, the variety $V_{\bba_1}$ reduces to  
\begin{align}
    \left\{(x_1,\ldots,x_{n+1})\in \mathbb{F}_q^{n+1} : x_{n+1}=\sum_{i=1}^{n}x_i^2-\sum_{i=1}^{n}2p_i x_i+\sum_{i=1}^{n}p_i^2\right\}, \label{eq_variety defined by P}
\end{align}
which is the $ n $-dimensional elliptical paraboloid from Example \ref{ex2}.

Given a set $\cP$ of points and a set $\cV$ of varieties, we define the number of incidences $I(\cP,\cV)$ between $\cP$ and $\cV$ as the cardinality of the set $\{(\bbp,V)\in \cP\times \cV:~\bbp\in V\}$. Our main result is as follows.

\begin{restatable}{theorem}{mainthmone}\label{main_thm1}
%\begin{theorem}
Let $\cP$ be a set of points in $\Fq^{n+d}$ and $\cV$ be a set of varieties of the form $V_{\bba_{1},\ldots,\bba_{d}}$ defined in (\ref{eq_variety_1}). Then, the number of incidences between $\cP$ and $\cV$ satisfies
$$\abs{I(\cP,\cV)-\frac{|\cP||\cV|}{q^d}}\leq q^{n/2}\sqrt{|\cP||\cV|}\parenv{1+\frac{|\cV|}{q}}^{1/2},$$ 
when $d\geq 2$ and 
$$\abs{I(\cP,\cV)-\frac{|\cP||\cV|}{q}}\leq q^{n/2}(1-\frac{1}{q})\sqrt{|\cP||\cV|},$$ 
when $d= 1$.
%\end{theorem}
\end{restatable}

Let $G$ be the bipartite incidence graph defined by all points in $\mathbb{F}_q^{n+d}$ and all varieties of the form $V{\bba_{1},\ldots,\bba_{d}}$. The incidence problem between points and varieties of form $V{\bba_{1},\ldots,\bba_{d}}$ was first studied by Phuong, Pham, Sang, Valculescu, and Vinh in \cite{PPSVV16}, where they established the following incidence bound by applying Lemma \ref{expander_mix_bipartite} to $G$.

\begin{theorem}\cite[Theorem 1.3]{PPSVV16}\label{Phuong's_thm}
Let $\cP$ be a set of points in $\Fq^{n+d}$ and $\cV$ be a set of varieties of the form $V_{\bba_{1},\ldots,\bba_{d}}$ defined in (\ref{eq_variety_1}). Then, the number of incidences between $\cP$ and $\cV$ satisfies
$$\abs{I(\cP,\cV)-\frac{|\cP||\cV|}{q^d}}\leq q^{dn/2}\sqrt{|\cP||\cV|}.$$ 
\end{theorem}

\begin{remark}\label{rmk1}
When $d = 1$, Theorem \ref{main_thm1} provides a factor of $(1-\frac{1}{q})$ improvement on the bound of $\abs{I(\cP,\cV) - \frac{|\cP||\cV|}{q^d}}$ compared to Theorem \ref{Phuong's_thm} in \cite{PPSVV16}. For $d \geq 2$, note that 
\[
q^{n/2}\parenv{1 + \frac{|\cV|}{q}}^{1/2} = q^{dn/2}\parenv{\frac{q + |\cV|}{q^{(d-1)n + 1}}}^{1/2}
\]
and $\frac{q + |\cV|}{q^{(d-1)n + 1}} < 1$ when $|\cV| < q^{(d-1)n + 1} - q$. Thus, compared to Theorem \ref{Phuong's_thm}, Theorem \ref{main_thm1} provides a factor of $q^{-\frac{(d-1)n}{2}}\parenv{1 + \frac{|\cV|}{q}}^{1/2}$ improvement on the bound of $\abs{I(\cP,\cV) - \frac{|\cP||\cV|}{q^d}}$ when $|\cV| < q^{(d-1)n + 1} - q$.
\end{remark}

\begin{remark}\label{rmk_comparsion}
Note that in order to prove Theorem \ref{Phuong's_thm}, the authors in \cite{PPSVV16} bound the third eigenvalue of the point-variety incidence matrix by estimating the number of paths of length $3$ in the graph and then applying Lemma \ref{expander_mix_bipartite}.

In contrast, to achieve our improved bound, we fully characterize \emph{all} the eigenvalues, along with their eigenspaces of the point-variety incidence matrix. This comprehensive characterization allows us to derive a sharper bound.
\end{remark}

\subsection{Application to point-flat incidences problem}
In \cite{LS14}, Lund and Saraf proved an incidence bound for points and blocks in a balanced incomplete block design (BIBD); see Definition 1.1 in \cite{Colbourn2010crc} for details. Specifically, in the same spirit as Vinh's work \cite{Vinh11}, they first determined the third eigenvalue of the adjacency matrix of the incidence graph defined by points and BIBDs, and then used Lemma \ref{expander_mix_bipartite} to obtain the incidence bound. As an application of their bound, they derived the following incidence bound for points and $n$-flats (i.e., $n$-dimensional affine subspaces) in $\mathbb{F}_q^{n+d}$, which was first proved by Haemers \cite[Chapter 3]{Haemers80}; see also \cite[Theorem 2.3]{BIP14_arxiv}.

\begin{theorem}\cite[Corollary 2]{LS14}\label{Lund's_thm}
Let $\cP$ be a set of points in $\Fq^{n+d}$ and $\cF$ be a set of $n$-flats in $\Fq^{n+d}$. Then, the number of incidences between $\cP$ and $\cF$ satisfies 
$$\abs{I(\cP,\cF)-\frac{|\cP||\cF|}{q^d}}\leq (1+o_{q}(1))q^{dn/2}\sqrt{|\cP||\cF|}.$$ 
\end{theorem}

As a consequence of Theorem \ref{main_thm1}, we have the following bounds on incidences between points and $n$-flats of the form $F_{\bba_1,\ldots,\bba_{d}}$ defined in Example \ref{ex1}.

\begin{corollary}\label{coro_incidence_pts_flats}
Let $\cP$ be a set of points in $\Fq^{n+d}$ and $\cF$ be a set of $n$-flats of the form $F_{\bba_{1},\ldots,\bba_{d}}$ defined in (\ref{eq_flat_1}). Then, the number of incidences between $\cP$ and $\cF$ satisfies
$$\abs{I(\cP,\cF)-\frac{|\cP||\cF|}{q^d}}\leq q^{n/2}\sqrt{|\cP||\cF|}\parenv{1+\frac{|\cF|}{q}}^{1/2}$$ 
when $d\geq 2$ and 
$$\abs{I(\cP,\cV)-\frac{|\cP||\cF|}{q}}\leq q^{n/2}(1-\frac{1}{q})\sqrt{|\cP||\cF|}$$ 
when $d= 1$.
\end{corollary}

For positive integers $n\geq k\geq 1$, let $\qbinom{n}{k}_q=\frac{\prod_{i=0}^{k-1}q^{n-i}-1}{\prod_{i=0}^{k-1}q^{k-i}-1}$ be the standard Gaussian binomial coefficient. As an application of Corollary \ref{coro_incidence_pts_flats}, we obtain the following incidence bound for points and $n$-flats in $\Fq^{n+d}$, which improves upon the result of Theorem \ref{Lund's_thm} for certain parameter regimes.

\begin{restatable}{theorem}{mainthmtwo}\label{main_thm2}
%\begin{theorem}
Let $\cP$ be a set of points in $\Fq^{n+d}$ and $\cF$ be a set of $n$-flats in $\Fq^{n+d}$. Then, the number of incidences between $\cP$ and $\cF$ satisfies
$$\abs{I(\cP,\cF)-\frac{|\cP||\cF|}{q^d}}\leq q^{n/2}\sqrt{|\cP||\cF|}\parenv{\frac{|\cF|}{q}+\frac{\qbinom{n+d}{d}_q}{q^{dn}}}^{1/2}$$ 
when $d\geq 2$ and
$$\abs{I(\cP,\cF)-\frac{|\cP||\cF|}{q}}\leq q^{n/2}(1-\frac{1}{q})\sqrt{|\cP||\cF|}\parenv{\frac{\qbinom{n+1}{1}_q}{q^{n}}}^{1/2}$$
when $d= 1$. 
%\end{theorem}
\end{restatable}

\begin{remark}\label{rmk2}

Note that $\qbinom{n+d}{d}_q = q^{dn}(1 + o_q(1))$. Therefore, when $d = 1$, Theorem \ref{main_thm2} provides a factor of $\left(1 - \frac{1}{q} - o\left(\frac{1}{q}\right)\right)$ improvement on the bound of $\abs{I(\cP, \cF) - \frac{|\cP||\cF|}{q}}$ compared to Theorem \ref{Lund's_thm} in \cite{BIP14_arxiv}.

For $d \geq 2$, using the estimate $\qbinom{n+d}{d}_q = q^{dn}(1 + o_q(1))$, we have
\begin{align*}
    q^{n/2}\parenv{\frac{|\cF|}{q} + \frac{\qbinom{n+d}{d}_q}{q^{dn}}}^{1/2} &= q^{dn/2}\parenv{\frac{|\cF|+q(1 + o_q(1))}{q^{(d-1)n+1}}}^{1/2},
\end{align*}
which is less than $q^{dn/2}$ when $|\cF| < q^{(d-1)n + 1} - q(1 + o_q(1))$. Thus, compared to Theorem \ref{Lund's_thm}, Theorem \ref{main_thm2} yields an improvement by a factor of $\parenv{\frac{|\cF| + q(1 + o_q(1))}{q^{(d - 1)n+1}}}^{\frac{1}{2}}$ in the bound on $\abs{I(\cP, \cV) - \frac{|\cP||\cV|}{q^d}}$ when $|\cF| < q^{(d-1)n + 1} - q(1 + o_q(1))$.

\end{remark}

\subsection{An application to the pinned distances problem}
Let $q$ be an odd prime power, and let $\cP$ be a set of points in $\Fq^n$. For a fixed point $\bby = (y_1, \ldots, y_n) \in \Fq^n$, the \emph{pinned distance set} between the \emph{pin} $\bby$ and the point set $\cP$ is defined as follows:
\begin{equation}\label{eq_pinned_distance}
    \Delta(\cP, \bby) \triangleq \left\{\sum_{i=1}^{n}(x_i - y_i)^2 : \bbx = (x_1, \ldots, x_n) \in \cP\right\}.
\end{equation}

In \cite{CEHIK12}, Chapman, Erdo{\u{g}}an, Hart, Iosevich, and Koh proved that a sufficiently large set of points determines many pinned distances for many different pins. Here, we refer to the following version from \cite{CILRR17}.

\begin{theorem}\cite[Corollary 1 and Corollary 2]{CILRR17}\label{pinned_distance_Chapman12}
    Let $\cP$ be a set of points in $\Fq^{n}$ such that $|\cP|\geq \frac{\sqrt{1-\varepsilon}}{\varepsilon}q^{\frac{n+1}{2}}$ for some $0<\varepsilon<1$. Then
    $$\frac{\sum_{\bby\in\cP}|\Delta(\cP,\bby)|}{|\cP|}>(1-\varepsilon)q.$$
    Moreover, there is a subset of points $\cQ\subseteq \cP$ such that $|\cQ|\geq (1-\sqrt{\varepsilon})|\cP|$ and $|\Delta(\cP,\bby)|>(1-\sqrt{\varepsilon})q$ holds for every $\bby\in \cQ$.
\end{theorem}

Using a similar approach as in \cite{CILRR17} (see also \cite{PPSVV16} and \cite{PTV17}), we apply Theorem \ref{main_thm1} to prove the result of Theorem \ref{pinned_distance_Chapman12} under a slightly weaker condition.

\begin{restatable}{corollary}{coropinneddistance}\label{pinned_distance_improved}
%\begin{corollary}
    Let $\cP$ be a set of points in $\Fq^{n}$ such that $|\cP|\geq \frac{\sqrt{1-\varepsilon}}{\varepsilon}(q^{\frac{n+1}{2}}-q^{\frac{n-1}{2}})$ for some $0<\varepsilon<1$. Then
    $$\frac{\sum_{\bby\in\cP}|\Delta(\cP,\bby)|}{|\cP|}\geq(1-\varepsilon)q.$$
    Moreover, there is a subset of points $\cQ\subseteq \cP$ such that $|\cQ|\geq (1-\sqrt{\varepsilon})|\cP|$ and $|\Delta(\cP,\bby)|\geq(1-\sqrt{\varepsilon})q$ holds for every $\bby\in \cQ$.
%\end{corollary}
\end{restatable}

\subsection{Outline}
The rest of the paper is organized as follows: In Section \ref{sec: pre}, we introduce some notations and preliminary results. In Section \ref{sec: proof of main thm}, we prove Theorem \ref{main_thm1} by giving a full characterization of the singular values and corresponding left and right-singular vectors of the incidence matrix of the point-variety incidence graph. Then, in Section \ref{sec: app}, we prove Theorem \ref{main_thm2} and Corollary \ref{pinned_distance_improved}. Finally, we conclude the paper in Section \ref{sec: conclusion}.  

\section{Preliminaries and notations}\label{sec: pre}

Let $q$ be a prime power and $\Fq$ be the finite field of size $q$. For a prime $p$, if $\Fq$ is a degree $m$ extension over $\mathbb{F}_p$, then the trace function $\mathrm{Tr}:\Fq\rightarrow \mathbb{F}_p$ is defined as
\begin{equation}\label{eq_trace}
\mathrm{Tr}(a)=a+a^p+\cdots+a^{p^{m-1}}.
\end{equation}
Due to the linearity of the Frobenius-homomorphism, the trace function is also linear, i.e., $\mathrm{Tr}(a+b)=\mathrm{Tr}(a)+\mathrm{Tr}(b)$. Moreover, it is well-known that the set of characters (group-homomorphisms) from $\Fq^n$ to $\mathbb{C}^{*}$ are defined by the vectors of $\Fq^n$ as follows: The character defined by $\bbv\in \Fq^{n}$ is
\begin{equation}\label{eq_character}
    \chi_{\bbv}(\bba)\triangleq e\left(\mathrm{Tr}(\langle \bbv,\bba \rangle)\right),
\end{equation}
where $\langle \bbv,\bba \rangle=\sum_{i=1}^{n}v_ia_i$ is the standard inner product between vectors $\bbv$ and $\bba$, and $e(x)\triangleq \mathrm{e}^{\frac{2\pi i x}{p}}$. For a vector $\bbv=(v_1,v_2,\ldots,v_n)\in \Fq^n$ and a subset $R\subseteq [n]$, denote $\bbv|_{R}$ as the vector obtained by projecting the coordinates of $\bbv$ onto $R$.

\subsection{The point-variety incidence graph}\label{sec_PV_incidence_graph}

Consider the following incidence bipartite graph $G=(A\cup B, E)$, derived from the point-variety incidences as follows:
\begin{itemize}
    \item The vertex parts are, $A=\{V_{\bba_{1},\ldots,\bba_{d}}: \bba_{i}\in \Fq^{n+1},~1\leq i\leq d\}$, the set of all varieties of the form $V_{\bba_{1},\ldots,\bba_{d}}$ in $\Fq^{n+d}$, and $B=\Fq^{n+d}$, the set of all points in $\Fq^{n+d}$. 
    \item There is an edge between some $V_{\bba_{1},\ldots,\bba_{d}}$ and $\bbp\in \Fq^{n+d}$ if and only if $\bbp\in V_{\bba_{1},\ldots,\bba_{d}}$.
\end{itemize}
By its definition in (\ref{eq_variety_1}), the variety $V_{\bba_1,\ldots,\bba_d}$ contains exactly $q^n$ distinct points in $\Fq^{n+d}$. Thus, each vertex in $A$ has degree $q^n$. Next, by the following lemma from \cite{PPSVV16}, we show that each vertex in $B$ has degree $q^{dn}$. For completeness, we include its proof here.
\begin{lemma}\cite[Lemma 3.1]{PPSVV16}\label{lem_number of varieties}
    For any two distinct $d$-tuple of vectors $(\bba_1,\ldots,\bba_d),(\tilde{\bba}_1,\ldots,\tilde{\bba}_d)\in (\Fq^{n+1})^{d}$, we have $V_{\bba_1,\ldots,\bba_d}\neq V_{\tilde{\bba}_1,\ldots,\tilde{\bba}_d}$.
\end{lemma}

\begin{proof}
    Since $(\bba_1, \ldots, \bba_d) \neq (\tilde{\bba}_1, \ldots, \tilde{\bba}_d)$, we assume w.l.o.g. that $\bba_1 \neq \tilde{\bba}_1$. Then, we have
    $$
    f_{\bba_1} - f_{\tilde{\bba}_1} = \sum_{j=1}^{n}(a_{1,j} - \tilde{a}_{1,j})x_j^{b_{1,j}} + (a_{1,n+1} - \tilde{a}_{1,n+1}).
    $$
    Assume that $b_{1,j} \equiv \tilde{b}_{1,j} \mod (q-1)$ for some integer $1 \leq \tilde{b}_{1,j} \leq q-1$. Then, $f_{\bba_1} - f_{\tilde{\bba}_1}$ has the same number of zeros as 
    $$
    \sum_{j=1}^{n}(a_{1,j} - \tilde{a}_{1,j})x_j^{\tilde{b}_{1,j}} + (a_{1,n+1} - \tilde{a}_{1,n+1}),
    $$
    which is a non-zero polynomial of total degree at most $q-1$ in $\Fq[x_1, \ldots, x_n]$. Thus, by the well-known Schwartz-Zippel Lemma (refer to Theorem 6.13 in \cite{lidl1997}), 
    \begin{align}\label{eq_lem_1}
        \abs{\{\bbp = (p_1, \ldots, p_n) \in \Fq^n : (f_{\bba_1} - f_{\tilde{\bba}_1})(\bbp) = 0\}} \leq (q-1)q^{n-1} < q^n.
    \end{align}
        
    On the other hand, by (\ref{eq_variety_1}), it holds that 
    $$
    V_{\bba_1, \ldots, \bba_d} = \{\parenv{\bbp, (h_1 + f_{\bba_1})(\bbp), \ldots, (h_d + f_{\bba_d})(\bbp)} : \bbp \in \Fq^n\}.
    $$
    If $V_{\bba_1, \ldots, \bba_d} = V_{\tilde{\bba}_1, \ldots, \tilde{\bba}_d}$, then we have $(f_{\bba_1} - f_{\tilde{\bba}_1})(\bbp) = 0$ for every $\bbp \in \Fq^n$, which contradicts (\ref{eq_lem_1}). This concludes the proof.
\end{proof}
Next, by applying Lemma \ref{lem_number of varieties} we show that the graph $G$ is indeed biregular. 
\begin{lemma}\label{vertex_degree_in_B}
Each vertex of $B$ in the graph $G$ has degree $q^{dn}$.
\end{lemma}
\begin{proof}
    Recall that for each point $\bbp \in \Fq^{n+d}$, as a vertex in $B$, its neighborhood in $G$ consists of the varieties of the form $V_{\bba_1,\ldots,\bba_d}$ such that $\bbp \in V_{\bba_1,\ldots,\bba_d}$. Thus, by Lemma \ref{lem_number of varieties}, the degree of $\bbp$ in the graph $G$ is equal to
    $$
    \abs{\left\{(\bba_1, \ldots, \bba_d) \in \Fq^{n+1} : \bbp \in V_{\bba_1,\ldots,\bba_d}\right\}}.
    $$    
    
    Meanwhile, by (\ref{eq_variety_0}) and (\ref{eq_variety_1}), $\bbp \in V_{\bba_1,\ldots,\bba_d}$ is equivalent to
    \begin{equation*}
    (p_{1}^{b_{i,1}}, \ldots, p_{n}^{b_{i,n}}, 1) \cdot 
    \left(\begin{array}{c}
        a_{i,1}   \\
        \vdots    \\
        a_{i,n}   \\
        a_{i,n+1} 
    \end{array}\right) = p_{n+i} - h_i(\bbp|_{[n]})
    \end{equation*}
    holding for every $i \in [d]$. Moreover, for each $i \in [d]$, there are exactly $q^n$ distinct $\bba_i \in \Fq^{n+1}$ satisfying the above linear equation. Therefore, we have 
    $$
    \abs{\left\{(\bba_1, \ldots, \bba_d) \in \Fq^{n+1} : \bbp \in V_{\bba_1,\ldots,\bba_d}\right\}} = q^{dn},
    $$
    which concludes the proof.
\end{proof}

Let $\bbT$ be the $q^{d(n+1)}\times q^{n+d}$ point-variety incidence matrix derived from the above bipartite graph. The rows of $\bbT$ are indexed by all the varieties of form $V_{\bba_{1},\ldots,\bba_{d}}$ and the columns of $\bbT$ are indexed by all the points in $\Fq^{n+d}$. Then, the entries of $\bbT$ satisfy
\begin{equation}\label{eq_incidence_matrix}
    \bbT(V_{\bba_{1},\ldots,\bba_{d}},\bbp)=\begin{cases}
    1,~\mathrm{if}~\bbp\in V_{\bba_{1},\ldots,\bba_{d}};\\
    0,~\mathrm{otherwise}.
    \end{cases}
\end{equation}

Throughout the paper, we use $\bbA'$ and $\bbA^{*}$ to denote the transpose and the Hermitian transpose of the matrix $\bbA$, respectively. Following a similar idea as that in \cite{Zachi23}, our main result will follow by determining the singular value decomposition (SVD) of the matrix $\bbT$. Then, we apply this decomposition in conjuction with an analysis similar to the expander mixing lemma \cite{AC88}. To this end, we will explore the eigenvalues and eigenvectors of the symmetric matrix $\bbT\cdot \bbT^*$ and $\bbT^*\cdot \bbT$.

\subsection{Singular value decomposition}

Let $\bbT$ be an $m\times n$ complex matrix and assume w.l.o.g. that $m\geq n$. The Singular Value Decomposition (SVD) of $\bbT$ is a factorization of the form $\bbT=\bbU\cdot \mathbf{\Sigma}\cdot \bbV^{*}$, where:
\begin{itemize}
    \item $\bbU$ is an $m\times m$ complex unitary matrix;
    \item $\mathbf{\Sigma}$ is an $m\times n$ rectangular diagonal matrix with non-negative real numbers on the diagonal;
    \item $\bbV$ is an $n\times n$ complex unitary matrix.
\end{itemize}
The nonzero elements $\sigma_{i,i}$ on the main diagonal of $\mathbf{\Sigma}$ are the singular values of $\bbT$. Then, the columns of $\bbU$ (left-singular vectors), denoted by
$\bbu_1',\ldots,\bbu_m'$, and the columns of $\bbV$ (right-singular vectors), denoted by $\bbv_1',\ldots,\bbv_n'$, form two sets of orthonormal bases respectively.

Writing
$\mathbf{\Sigma}=\left(\begin{array}{c}
     \bbD  \\
     \mathbf{0}_{(m-n)\times n}
\end{array}\right)$,
where $\bbD$ is an $n\times n$ diagonal matrix and $\mathbf{0}_{(m-n)\times n}$ is an $(m-n)\times n$ zero matrix, we have
\begin{equation}\label{eq_TT^*}
    \bbT\cdot \bbT^*=\bbU\cdot \left(\begin{array}{cc}
     \bbD^2 & \mathbf{0}_{n\times (m-n)}  \\
     \mathbf{0}_{(m-n)\times n} & \mathbf{0}_{(m-n)\times (m-n)}
\end{array}\right)\cdot \bbU^*,
\end{equation}
and 
\begin{equation}\label{eq_T^*T}
    \bbT^*\cdot \bbT=\bbV\cdot \bbD^2 \cdot \bbV^*.
\end{equation}
Clearly, both $\bbT\cdot \bbT^*$ and $\bbT^*\cdot \bbT$ are symmetric positive semi-definite matrices and have the same positive eigenvalues. Furthermore, a positive $\lambda$ is an eigenvalue of $\bbT^{*}\cdot \bbT$ corresponding to an eigenspace of dimension $d$ if and only if it is also an eigenvalue of $\bbT\cdot \bbT^*$ corresponding to an eigenspace of the same dimension $d$. This equivalence holds if and only if the matrix $\bbT$ has exactly $d$ singular values equal to $\sqrt{\lambda}$.

\subsection{Group algebra}

For the abelian group $\Fq^n$, its group algebra over $\mathbb{C}$, denoted by $\mathbb{C}[\Fq^n]$, consists of all possible formal sums of the variables $x_{\bbu}$, $\bbu\in \Fq^n$, over $\mathbb{C}$. That is
\begin{equation*}
    \mathbb{C}[\Fq^n]=\left\{\sum_{\bbu\in\Fq^n}c_{\bbu}x_{\bbu}:~c_{\bbu}\in \mathbb{C}\right\}.
\end{equation*}
For elements $\sum_{\bbu\in\Fq^n}c_{\bbu}x_{\bbu}$ and $\sum_{\bbu\in\Fq^n}\tilde{c}_{\bbu}x_{\bbu}$, their addition in this algebra is defined component-wise,
\begin{equation*}
    \sum_{\bbu\in\Fq^n}c_{\bbu}x_{\bbu}+\sum_{\bbu\in\Fq^n}\tilde{c}_{\bbu}x_{\bbu}\triangleq\sum_{\bbu\in\Fq^n}(c_{\bbu}+\tilde{c}_{\bbu})x_{\bbu},
\end{equation*}
while the multiplication is defined via the addition in the group $\Fq^n$:
\begin{equation*}
    \sum_{\bbu\in\Fq^n}c_{\bbu}x_{\bbu}\cdot \sum_{\bbu\in\Fq^n}\tilde{c}_{\bbu}x_{\bbu}\triangleq\sum_{\bbu\in\Fq^n}\left(\sum_{\tilde{\bbu}\in \Fq^n}c_{\bbu-\tilde{\bbu}}\tilde{c}_{\tilde{\bbu}}\right)x_{\bbu}.
\end{equation*}

Each element $f=\sum_{\bbu\in\Fq^n}c_{\bbu}x_{\bbu}\in \mathbb{C}[\Fq^n]$ can be viewed as a function $f:\Fq^n\rightarrow \mathbb{C}$ (or as a vector of length $q^n$, indexed by the elements of $\Fq^n$), by simply setting $f(\bbu)=c_{\bbu}$ for any $\bbu\in \Fq^n$. Thus, the group algebra $\mathbb{C}[\Fq^n]$ is a vector space of dimension $q^n$ over $\mathbb{C}$, and the elements $\{x_{\bbu}\}_{\bbu\in \Fq^n}$ form a basis called the standard basis. Another important basis of $\mathbb{C}[\Fq^n]$ is defined by the set of characters of $\Fq^n$
\begin{equation*}
    \left\{\sum_{\bbu\in\Fq^n}\chi_{\bbv}(\bbu)x_{\bbu}\right\}_{\bbv\in\Fq^n}.
\end{equation*}
In the following, if there is no confusion, we will also use $\chi_{\bbv}$ to refer to $\sum_{\bbu\in\Fq^n}\chi_{\bbv}(\bbu)x_{\bbu}$.
Indeed, the characters are linearly independent as they are orthogonal under the inner product defined as
\begin{equation*}
    \langle f,g \rangle=\sum_{\bbu\in \Fq^n}f(\bbu)\overline{g(\bbu)},
\end{equation*}
for any $f,g\in \mathbb{C}[\Fq^n]$ and they form a basis of $\mathbb{C}[\Fq^n]$ since the number of characters equals the dimension of the vector space.

\begin{definition}[Multiplication operators of the group algebra]\label{def_Group_algebra_multiplication}
    Given an element $z\in\mathbb{C}[\Fq^n]$, we define the linear operator $\phi_{z}:\mathbb{C}[\Fq^n]\rightarrow \mathbb{C}[\Fq^n]$ as follows:
    $$\phi_{z}(f)\triangleq z\cdot f,$$
    for any $f\in \mathbb{C}[\Fq^n]$. Moreover, since the group $\Fq^{n}$ is commutative, we have $\phi_{z}(f)= f\cdot z$.
\end{definition}

For the proof of our results, we also need the following well-known result. For completeness, we include its proof here as it is given, e.g., in \cite{Zachi23}.

\begin{proposition}\label{prop_character_eigenvector}\cite[Proposition 2.4]{Zachi23} 
Let $z\in\mathbb{C}[\Fq^n]$. Then, any character $\chi_{\bbv}$ of $\Fq^n$ is an eigenvector of $\phi_z$ with eigenvalue $\langle z,\chi_{\bbv} \rangle$
\end{proposition}

\begin{proof}
    The result follows from the following calculation:
    \begin{align*}
        \phi_z(\chi)&=z\cdot \chi=\left(\sum_{\bbv\in \Fq^n}z_{\bbv}x_{\bbv}\right)\left(\sum_{\tilde{\bbv}\in \Fq^n}\chi(\tilde{\bbv})x_{\tilde{\bbv}}\right)\\
        &=\sum_{\bbv\in \Fq^n}\sum_{\tilde{\bbv}\in \Fq^n}z_{\tilde{\bbv}}\chi(\bbv-\tilde{\bbv})x_{\bbv}\\
        &=\sum_{\bbv\in \Fq^n}\sum_{\tilde{\bbv}\in \Fq^n}z_{\tilde{\bbv}}\chi(\bbv)\overline{\chi(\tilde{\bbv})}x_{\bbv}\\
        &=\sum_{\bbv\in \Fq^n}\langle z,\chi \rangle \chi(\bbv)x_{\bbv}=\langle z,\chi \rangle\cdot \chi.
    \end{align*}
\end{proof}

\section{Proof of the main results}\label{sec: proof of main thm}

In this section, we present the proof of our main result, Theorem \ref{main_thm1}. To this end, we first show that the matrices $ \bbT^{*} \cdot \bbT $ and $ \bbT \cdot \bbT^{*} $ correspond to certain multiplication operators in two distinct group algebras. Then, by applying Proposition \ref{prop_character_eigenvector}, we fully characterize their eigenvalues and eigenvectors in Theorems \ref{thm_eigenvalue_T^*T} and \ref{thm_eigenvalue_TT^*}. Building on these characterizations, we demonstrate that for any point set $ \cP $ and variety set $ \cV $ of the form \eqref{eq_variety_1}, their corresponding indicator vectors $ \mathbf{1}_{\cP} $ and $ \mathbf{1}_{\cV} $ have substantial projections onto the zero eigenspace. Finally, employing an approach similar to the proof of Lemma \ref{expander_mix_bipartite}, we establish Theorem \ref{main_thm1}.

\subsection{Right singular vectors of the incidence matrix}\label{sec: right singular vectors}

In this section, we give a complete description of the right singular vectors of $\bbT$ by considering the matrix $\bbA= \bbT^{*}\cdot \bbT$.

By (\ref{eq_incidence_matrix}) and (\ref{eq_T^*T}), $\bbA$ is a square matrix of order $q^{n+d}$, whose rows and columns  are indexed by the points of $\Fq^{n+d}$. Moreover, for points $\bbu,\bbv\in \Fq^{n+d}$, the entry $\bbA_{\bbu,\bbv}$ equals to the number of varieties $V_{\bba_1,\ldots,\bba_d}$ containing both $\bbu$ and $\bbv$. Thus, by the properties of the incidence bipartite graph $G$, we can calculate $\bbA_{\bbu,\bbv}$ as follows.

\begin{lemma}\label{lem_entry_T^*T}
For any two points $\bbu,\bbv\in \Fq^{n+d}$, it holds that
\begin{equation*}
    \bbA_{\bbu,\bbv}=\left\{\begin{array}{cc}
        q^{dn}, & \mathrm{if}~\bbu=\bbv; \\
        q^{d(n-1)}, & \mathrm{if}~\bbu|_{[n]}\neq\bbv|_{[n]};\\
        0, & \mathrm{otherwise}.
    \end{array}\right.
\end{equation*}   
\end{lemma}

\begin{proof}
    Assume that $\bbu=(u_1,\ldots,u_{n+d})$ and $\bbv=(v_1,\ldots,v_{n+d})$. When $\bbu=\bbv$, the result follows directly by the fact that the point-variety incidence graph $G$ defined in Section \ref{sec_PV_incidence_graph} is biregular with each vertex in $B$ having degree $q^{dn}$. Next, we count the number of varieties $V_{\bba_1,\ldots,\bba_d}$ containing both $\bbu$ and $\bbv$ for the case when $\bbu\neq\bbv$.

    By the definition of $V_{\bba_1,\ldots,\bba_d}$ in (\ref{eq_variety_0}) and (\ref{eq_variety_1}), $\bbu\in V_{\bba_1,\ldots,\bba_d}$ is equivalent to
    \begin{equation*}
    (u_{1}^{b_{i,1}},\ldots,u_{n}^{b_{i,n}},1)\cdot 
    \left(\begin{array}{c}
        a_{i,1}   \\
        \vdots    \\
        a_{i,n}   \\
        a_{i,n+1} 
    \end{array}\right)=u_{n+i}-h_i(\bbu|_{[n]})
    \end{equation*}
    holds for every $1\leq i\leq d$. Thus, if both $\bbu$ and $\bbv$ are contained in $V_{\bba_1,\ldots,\bba_d}$, then it holds for every $1\leq i\leq d$ that
    \begin{equation}\label{eq1_lem_entry_T^*T}
    \left(\begin{array}{cccc}
    u_{1}^{b_{i,1}} & \ldots & u_{n}^{b_{i,n}} & 1 \\
    v_{1}^{b_{i,1}} & \ldots & v_{n}^{b_{i,n}} & 1 
    \end{array}\right)\cdot 
    \left(\begin{array}{c}
        a_{i,1}   \\
        \vdots    \\
        a_{i,n}   \\
        a_{i,n+1} 
    \end{array}\right)=\left(\begin{array}{c}
        u_{n+i}-h_i(\bbu|_{[n]})   \\
        v_{n+i}-h_i(\bbv|_{[n]})    
    \end{array}\right).
    \end{equation}
    Since the $b_{i,j}$'s are fixed positive integers satisfying $\mathrm{gcd}(b_{i,j},q-1)=1$, then 
 the equality  $\alpha^{b_{i,j}}=\beta^{b_{i,j}}$ for $\alpha,\beta\in \Fq$  implies that $\alpha=\beta$. Thus, when $\bbu|_{[n]}\neq\bbv|_{[n]}$, 
    $$
    \left(\begin{array}{cccc}
    u_{1}^{b_{i,1}} & \ldots & u_{n}^{b_{i,n}} & 1 \\
    v_{1}^{b_{i,1}} & \ldots & v_{n}^{b_{i,n}} & 1 
    \end{array}\right)
    $$
    has rank $2$ over $\Fq$ for every $1\leq i\leq d$. This implies that for each $1\leq i\leq d$, there are exactly $q^{n-1}$ different choices of $(a_{i,1},\ldots,a_{i,n+1})$ satisfying (\ref{eq1_lem_entry_T^*T}). Therefore, there are $q^{d(n-1)}$ different varieties $V_{\bba_1,\ldots,\bba_d}$ containing both $\bbu$ and $\bbv$.

    When $\bbu\neq\bbv$ but $\bbu|_{[n]}=\bbv|_{[n]}$, there is an $1\leq i\leq d$ for which  $(u_{1}^{b_{i,1}},\ldots,u_{n}^{b_{i,n}},1)=(v_{1}^{b_{i,1}},\ldots,v_{n}^{b_{i,n}},1)$, but $u_{n+i}-h_i(\bbu|_{[n]})\neq v_{n+i}-h_i(\bbv|_{[n]})$ cannot hold since  $h_i(\bbu|_{[n]})=h_i(\bbv|_{[n]})$. 
    Therefore, there is no $V_{\bba_1,\ldots,\bba_d}$ containing both $\bbu$ and $\bbv$.
\end{proof}

The following lemma shows that the matrix $\bbA$ can be viewed as a representation of some multiplication operator under the standard basis of $\mathbb{C}[\Fq^{n+d}]$.

\begin{lemma}\label{lem_representation_T^*T}
    Let $f=\sum_{\bbu\in\Fq^{n+d}}c_{\bbu}x_{\bbu}\in \mathbb{C}[\Fq^{n+d}]$, where
    \begin{equation*}
        c_{\bbu}=\left\{\begin{array}{cc}
        q^{dn}, & \mathrm{if}~\bbu=\mathbf{0};\\
        q^{d(n-1)}, & \mathrm{if}~\bbu|_{[n]}\neq\mathbf{0};\\
        0, & \mathrm{otherwise}.
    \end{array}\right.
    \end{equation*}
    Then, $\bbA$ is the representation of $\phi_f$ under the standard basis $B=\{x_{\bbu}:\bbu\in \Fq^{n+d}\}$, i.e., $[\phi_{f}]_{B}^{B}=\bbA$.
\end{lemma}

\begin{proof}
Let $x_{\bbv}\in B$ be a standard basis vector, then by definition
\begin{align*}
    \phi_f(x_{\bbv})&=f\cdot x_{\bbv}=\sum_{\bbu\in\Fq^{n+d}}c_{\bbu-\bbv}x_{\bbu}\\   &=\sum_{\bbu\in\Fq^{n+d}}\bbA_{\bbu,\bbv}x_{\bbu},
\end{align*}
and  the result follows from Lemma \ref{lem_entry_T^*T}.
\end{proof}

By Proposition \ref{prop_character_eigenvector}, the characters of $\Fq^{n+d}$ form a basis of the eigenspaces of $\phi_f$. Next, we proceed by calculating their corresponding eigenvalues.

\begin{proposition}\label{prop_eigenvalue_T^*T}
    A character $\chi_{\bbv}\in \mathbb{C}[\Fq^{n+d}]$ of $\Fq^{n+d}$ is an eigenvector of $\phi_f$ with eigenvalue 
    \begin{equation*}
        \lambda=\left\{\begin{array}{cc}
        q^{(d+1)n}, & \mathrm{if}~\bbv=\mathbf{0};\\
        q^{dn}, & \mathrm{if}~\bbv|_{[n+1,n+d]}\neq\mathbf{0};\\
        0, & \mathrm{otherwise}. 
    \end{array}\right.
    \end{equation*}
\end{proposition}

\begin{proof}
    By Proposition \ref{prop_character_eigenvector}, $\lambda=\langle f,\chi_{\bbv} \rangle$. Then, by the definition of $f$, we have
    \begin{align}
        \bar{\lambda}& %=\overline{\langle f,\chi_{\bbv} \rangle}
        =\sum_{\bbu\in \Fq^{n+d}}\overline{f(\bbu)}\chi_{\bbv}(\bbu)=\sum_{\bbu\in \Fq^{n+d}}c_{\bbu}\chi_{\bbv}(\bbu) \nonumber\\
        &=q^{dn}+q^{d(n-1)}\sum_{\bbu\in \Fq^{n+d}\atop \bbu|_{[n]}\neq \mathbf{0}}e\left(\mathrm{Tr}(\langle \bbu,\bbv \rangle)\right).\label{eq1_prop_eigenvalue_T^*T}
    \end{align}
    If $\bbv=\mathbf{0}$, then $e(\mathrm{Tr}(\langle \bbu,\bbv \rangle))=1$ for any $\bbu\in \Fq^{n+d}$ with $\bbu|_{[n]}\neq \mathbf{0}$ and (\ref{eq1_prop_eigenvalue_T^*T}) becomes
    \begin{align*}
        q^{dn}+q^{d(n-1)}(q^{n+d}-q^d)=q^{(d+1)n}.
    \end{align*}
    Notice that $e(\mathrm{Tr}(\langle \bbu,\bbv \rangle))=e\left(\mathrm{Tr}\left(\sum_{i\in [n]}u_iv_i\right)\right) e\left(\mathrm{Tr}\left(\sum_{i=n+1}^{n+d}u_iv_i\right)\right)$. Thus, if $\bbv|_{[n+1,n+d]}\neq \mathbf{0}$, then (\ref{eq1_prop_eigenvalue_T^*T}) becomes
    \begin{align*}
       & q^{dn}+q^{d(n-1)}\sum_{\bbu\in \Fq^{n+d}\atop \bbu|_{[n]}\neq \mathbf{0}} e\left(\mathrm{Tr}\left(\sum_{i\in [n]}u_iv_i\right)\right) e\left(\mathrm{Tr}\left(\sum_{i=n+1}^{n+d}u_iv_i\right)\right)\\
     = & q^{dn}+q^{d(n-1)}\sum_{\bbu|_{[n]}\in \Fq^{n}\setminus\{\mathbf{0}\}} e\left(\mathrm{Tr}\left(\sum_{i\in [n]}u_iv_i\right)\right) \sum_{\bbu|_{[n+1,n+d]}\in \Fq^{d}}e\left(\mathrm{Tr}\left(\sum_{i=n+1}^{n+d}u_iv_i\right)\right)\\
     = & q^{dn}+q^{d(n-1)}\sum_{\bbu|_{[n]}\in \Fq^{n}\setminus\{\mathbf{0}\}} e\left(\mathrm{Tr}\left(\sum_{i\in [n]}u_iv_i\right)\right) \cdot 0=q^{dn}.
    \end{align*}
    If $\bbv|_{[n]}\neq \mathbf{0}$ and $\bbv|_{[n+1,n+d]}= \mathbf{0}$, then (\ref{eq1_prop_eigenvalue_T^*T}) becomes
    \begin{align*}
       & q^{dn}+q^{d(n-1)}\sum_{\bbu\in \Fq^{n+d}\atop \bbu|_{[n]}\neq \mathbf{0}} e\left(\mathrm{Tr}\left(\sum_{i\in [n]}u_iv_i\right)\right) e\left(\mathrm{Tr}\left(\sum_{i=n+1}^{n+d}u_iv_i\right)\right)\\
     = & q^{dn}+q^{d(n-1)}\sum_{\bbu|_{[n]}\in \Fq^{n}\setminus\{\mathbf{0}\}} e\left(\mathrm{Tr}\left(\sum_{i\in [n]}u_iv_i\right)\right) \sum_{\bbu|_{[n+1,n+d]}\in \Fq^{d}} 1\\
     = & q^{dn}+q^{dn}\sum_{\bbu|_{[n]}\in \Fq^{n}\setminus\{\mathbf{0}\}} e\left(\mathrm{Tr}\left(\sum_{i\in [n]}u_iv_i\right)\right)=q^{dn}-q^{dn}=0.
    \end{align*}
    This concludes the proof.
\end{proof}

The following theorem follows easily from Proposition \ref{prop_eigenvalue_T^*T}.

\begin{theorem}\label{thm_eigenvalue_T^*T}
    The characters of $\Fq^{n+d}$ are eigenvectors of the operator $\phi_f$ with the eigenvalues listed in Table \ref{Table_1}.
    \begin{table}[h!]
    \caption{Eigenvalues and eigenvectors of $\phi_f$}\label{Table_1}
    \centering
    \begin{tabular}{|c|c|c|}
    \hline
    Character $\chi_{\mathbf{v}}$ & No. of  characters & Eigenvalue \\ 
    \hline
    $\mathbf{v}=\mathbf{0}$ & $1$ & $q^{(d+1)n}$ \\ \hline
    $\mathbf{v}|_{[n+1,n+d]}\neq\mathbf{0}$ & $q^{n+d}-q^n$ & $q^{dn}$ \\ \hline
    $\mathbf{v}\neq\mathbf{0}$, $\mathbf{v}|_{[n+1,n+d]}=\mathbf{0}$ & $q^n-1$ & $0$ \\ \hline
    \end{tabular}
\end{table}
\end{theorem}

\begin{corollary}\label{coro_right_singular_T}
    The normalized characters of $\Fq^{n+d}$ serve as right-singular vectors of the matrix $\bbT$, each associated with a singular value equals to the square root of its corresponding eigenvalue listed in Theorem \ref{thm_eigenvalue_T^*T}.
\end{corollary}

Building on Theorem \ref{thm_eigenvalue_T^*T}, we next establish a bound on the squared norm of the projection of an indicator vector onto the eigenspaces. 

\begin{lemma}\label{lem_proj_on_V}
    Let $\mathbf{1}_{\cP}\in \mathbb{C}^{q^{n+d}}$ be the indicator vector of a set of points $\cP\subseteq \Fq^{n+d}$. Then, the squared norm of the projection of $\mathbf{1}_{\cP}$ onto the eigenspaces with eigenvalues $q^{(d+1)n}$ and $q^{dn}$ are $\frac{|\cP|^2}{q^{n+d}}$ and at most $(1-\frac{1}{q^d})|\cP|$, respectively.
\end{lemma}

\begin{proof}
    The squared norm of the projection of $\mathbf{1}_{\cP}$ onto the normalized trivial character is
    $$\abs{\langle \mathbf{1}_{\cP},\frac{\chi_{\mathbf{0}}}{q^{(n+d)/2}} \rangle}^2=\frac{|\cP|^2}{q^{n+d}}.$$

    Next, we bound the squared norm of the projection of $\mathbf{1}_{\cP}$ onto the eigenspace corresponding to the eigenvalue $q^{dn}$.
    \begin{align*}
       \sum_{\bbv\in\Fq^{n+d} \atop \bbv|_{[n+1,n+d]}\neq \mathbf{0}}\abs{\langle \mathbf{1}_{\cP},\frac{\chi_{\bbv}}{q^{(n+d)/2}}\rangle}^2 & =q^{-(n+d)}\sum_{\bbv\in\Fq^{n+d} \atop \bbv|_{[n+1,n+d]}\neq \mathbf{0}}\sum_{\bbu\in \cP}\chi_{\bbv}(\bbu)\sum_{\tilde{\bbu}\in \cP}\overline{\chi_{\bbv}(\tilde{\bbu})}\\
     & = q^{-(n+d)}\sum_{\bbv\in\Fq^{n+d} \atop \bbv|_{[n+1,n+d]}\neq \mathbf{0}}\sum_{\bbu,\tilde{\bbu}\in \cP}\chi_{\bbv}(\bbu-\tilde{\bbu})\\
     & = \frac{q^{n+d}-q^n}{q^{n+d}}|\cP|+q^{-(n+d)}\sum_{\bbv\in\Fq^{n+d} \atop \bbv|_{[n+1,n+d]}\neq \mathbf{0}}\sum_{\bbu\neq \tilde{\bbu}\in \cP}\chi_{\bbv}(\bbu-\tilde{\bbu})\\
     & = (1-\frac{1}{q^{d}})|\cP|+q^{-(n+d)}\sum_{\bbu\neq \tilde{\bbu}\in \cP}\sum_{\bbv\in\Fq^{n+d} \atop \bbv|_{[n+1,n+d]}\neq \mathbf{0}}e\left(\mathrm{Tr}\left(\langle \bbv,\bbu-\tilde{\bbu}\rangle\right)\right).
    \end{align*}
    Note that for any fixed $\bbu\neq \tilde{\bbu}\in \cP$, 
    \begin{align*}
        &\sum_{\bbv\in\Fq^{n+d} \atop \bbv|_{[n+1,n+d]}\neq \mathbf{0}}e\left(\mathrm{Tr}\left(\langle \bbv,\bbu-\tilde{\bbu}\rangle\right)\right)\\
        = & \sum_{\bbv_1\in\Fq^{d}\setminus\{\mathbf{0}\}}e\left(\mathrm{Tr}\left(\langle \bbv_1,(\bbu-\tilde{\bbu})|_{[n+1,n+d]}\rangle\right)\right)\sum_{\bbv_2\in\Fq^{n}}e\left(\mathrm{Tr}\left(\langle \bbv_2,(\bbu-\tilde{\bbu})|_{[n]}\rangle\right)\right),
    \end{align*}
    where $\bbv_1=\bbv|_{[n+1,n+d]}$ and $\bbv_2=\bbv|_{[n]}$.
    Thus, we have 
    \begin{equation*}
        \sum_{\bbv\in\Fq^{n+d} \atop \bbv|_{[n+1,n+d]}\neq \mathbf{0}}e\left(\mathrm{Tr}\left(\langle \bbv,\bbu-\tilde{\bbu}\rangle\right)\right)=\left\{\begin{array}{cc}
        0, & \mathrm{if}~\bbu|_{[n]}\neq\tilde{\bbu}|_{[n]};\\
        -q^{n}, & \mathrm{if}~\bbu\neq \tilde{\bbu}~\mathrm{but}~\bbu|_{[n]}=\tilde{\bbu}|_{[n]}.
    \end{array}\right.
    \end{equation*}
    This implies  that $\sum_{\bbu\neq \tilde{\bbu}\in \cP}\sum_{\bbv\in\Fq^{n+d} \atop \bbv|_{[n+1,n+d]}\neq \mathbf{0}}e\left(\mathrm{Tr}\left(\langle \bbv,\bbu-\tilde{\bbu}\rangle\right)\right)$ is a non-positive integer and this concludes the result.
\end{proof}

\subsection{Left singular vectors of the incidence matrix}\label{sec: left singular vectors}

In this section we take a similar approach to Section \ref{sec: right singular vectors} and provide a complete description of the left singular vectors of $\bbT$ by considering the matrix $\bbA=\bbT\cdot \bbT^{*}$.

By (\ref{eq_incidence_matrix}) and (\ref{eq_TT^*}), we know that $\bbA$ is a square matrix of order $q^{d(n+1)}$ and the rows and columns of $\bbA$ are indexed by the varieties $V_{\bba_1,\ldots,\bba_d}$ defined in (\ref{eq_variety_1}). Thus, by the properties of the incidence bipartite graph $G$, one can easily verify that for any two varieties $V_{\bba_1,\ldots,\bba_d},V_{\tilde{\bba}_1,\ldots,\tilde{\bba}_d}$,
\begin{align}
\bbA_{V_{\bba_1,\ldots,\bba_d},V_{\tilde{\bba}_1,\ldots,\tilde{\bba}_d}}&=|\{\bbu\in \Fq^{n+d}: \bbu\in V_{\bba_1,\ldots,\bba_d}\cap V_{\tilde{\bba}_1,\ldots,\tilde{\bba}_d}\}| \nonumber \\
&=\abs{\left\{\bbu\in \Fq^{n+d}: \begin{array}{c}
     u_{n+i}=h_i(\bbu|_{[n]})+f_{\bba_i}(\bbu|_{[n]})  \\
     ~~~~~~~~=h_i(\bbu|_{[n]})+f_{\tilde{\bba}_i}(\bbu|_{[n]}),\\
     \forall~1\leq i\leq d
\end{array}\right\}} \nonumber\\
&=\abs{\left\{\tilde{\bbu}\in \Fq^{n}: f_{\bba_i}(\tilde{\bbu})=f_{\tilde{\bba}_i}(\tilde{\bbu}),~\forall~1\leq i\leq d\right\}}. \label{eq_entry_TT^*}
\end{align}

For $d$ fixed vectors $\bbb_1,\ldots,\bbb_d$ in $(\mathbb{Z}^{+})^n$ with $\mathrm{gcd}(b_{i,j},q-1)=1$ for all $1\leq i\leq d$ and $1\leq j\leq n$, by the definition of $f_{\bba_i}(\bbx)$ in (\ref{eq_variety_0}), the set $\{f_{\bba_i}:\bba_i\in \Fq^{n+1}\}$ forms an additive subgroup of $\Fq[x_1,\ldots,x_n]$ that is isomorphic to $\Fq^{n+1}$ for each $1\leq i\leq d$. That is, the group 
$$G_{\bbb_1,\ldots,\bbb_d}\triangleq \{(f_{\bba_1},f_{\bba_2},\ldots,f_{\bba_d}): \bba_i\in \Fq^{n+1},~1\leq i\leq d \}$$
is isomorphic to $\Fq^{d(n+1)}$ (under the addition in $(\Fq[x_1,\ldots,x_n])^{d}$). Moreover, their corresponding group algebra $\mathbb{C}[G_{\bbb_1,\ldots,\bbb_d}]$ and $\mathbb{C}[\Fq^{d(n+1)}]$ are also isomorphic. Consequently, in the sequel, we will refer to them interchangeably as the same mathematical object.

The following lemma shows that the matrix $\bbA$ can be viewed as a representation of some multiplication operator under the standard basis of $\mathbb{C}[G_{\bbb_1,\ldots,\bbb_d}]$.

\begin{lemma}\label{lem_representation_TT^*}
    Let $z=\sum_{(f_1,\ldots,f_d)\in G_{\bbb_1,\ldots,\bbb_d}}c_{(f_1,\ldots,f_d)}x_{(f_1,\ldots,f_d)}\in \mathbb{C}[G_{\bbb_1,\ldots,\bbb_d}]$, where
    \begin{equation*}
        c_{(f_1,\ldots,f_d)}=\abs{\{(u_1,\ldots,u_n)\in \Fq^n: f_i(u_1,\ldots,u_n)=0,~\forall~1\leq i\leq d\}}.
    \end{equation*}
    Then, $\bbA$ is the representation of $\phi_z$ under the standard basis $B=\{x_{(f_1,\ldots,f_d)}:(f_1,\ldots,f_d)\in G_{\bbb_1,\ldots,\bbb_d}\}$, i.e., $[\phi_{z}]_{B}^{B}=\bbA$.
\end{lemma}

\begin{proof}
Let $x_{(f_1,\ldots,f_d)}\in B$ be a standard basis vector, then by definition
%\begin{align*}
$$
    \phi_z(x_{(f_1,\ldots,f_d)})=z\cdot x_{(f_1,\ldots,f_d)}
    =\sum_{(\tilde{f}_1,\ldots,\tilde{f}_d)\in G_{\bbb_1,\ldots,\bbb_d}}c_{(\tilde{f}_1-f_1,\ldots,\tilde{f}_d-f_d)}x_{(\tilde{f}_1,\ldots,\tilde{f}_d)}.
$$
%\end{align*}
Since for any $(f_1,\ldots,f_d),(\tilde{f}_1,\ldots,\tilde{f}_d)\in G_{\bbb_1,\ldots,\bbb_d}$, we can write $f_i=f_{\bba_i}$ and $\tilde{f}_i=f_{\tilde{\bba}_i}$ for some $\bba_i,\tilde{\bba}_i\in \Fq^{n+1}$. Then, by (\ref{eq_entry_TT^*}) and the definition of variety $V_{\bba_1,\ldots,\bba_d}$ in (\ref{eq_variety_1}),
we have
\begin{align*}
    &c_{(\tilde{f}_1-f_1,\ldots,\tilde{f}_d-f_d)}\\
   = &\abs{\{(u_1,\ldots,u_n)\in \Fq^n: (\tilde{f}_i-f_i)(u_1,\ldots,u_n)=0,~\forall~1\leq i\leq d\}}\\
   = &\abs{\{(u_1,\ldots,u_{n})\in \Fq^{n}: f_{\bba_i}(u_1,\ldots,u_n)=f_{\tilde{\bba}_i}(u_1,\ldots,u_n),~\forall~1\leq i\leq d\}}\\
   = & \abs{\{\bbu=(u_1,\ldots,u_{n+d})\in \Fq^{n+d}: \bbu\in V_{\bba_1,\ldots,\bba_d}\cap V_{\tilde{\bba}_1,\ldots,\tilde{\bba}_d}\}}.
\end{align*}
This concludes the proof.
\end{proof}

Having established that the matrix $\bbA$ represents the linear operator $\phi_z$ under the standard basis, we can conclude by Proposition \ref{prop_character_eigenvector} that the characters of $G_{\bbb_1,\ldots,\bbb_d}$ are the eigenvectors of $\bbA$. Since $G_{\bbb_1,\ldots,\bbb_d}\cong \Fq^{d(n+1)}$, we can associate each character of $G_{\bbb_1,\ldots,\bbb_d}$ with a unique character $\chi_{\bba_1,\ldots,\bba_d}$ of $\Fq^{d(n+1)}$, where $\bba_i\in \Fq^{n+1}$ for all $1\leq i\leq d$. Then, for each $(f_{\tilde{\bba}_1},\ldots,f_{\tilde{\bba}_d})\in G_{\bbb_1,\ldots,\bbb_d}$, we have $$\chi_{\bba_1,\ldots,\bba_d}\left(\left(f_{\tilde{\bba}_1},\ldots,f_{\tilde{\bba}_d}\right)\right)=e\left(\mathrm{Tr}\left(\sum_{i=1}^{d}\langle \bba_i,\tilde{\bba}_i \rangle\right)\right).$$
The following proposition gives their eigenvalues.

\begin{proposition}\label{prop_eigenvalue_TT^*}
    A character $\chi_{\bba_1,\ldots,\bba_d}\in\mathbb{C}[G_{\bbb_1,\ldots,\bbb_d}]$ of $G_{\bbb_1,\ldots,\bbb_d}$ is an eigenvector of $\phi_z$ with eigenvalue 
    
    \begin{equation*}
        \lambda=\left\{\begin{array}{cc}
        q^{(d+1)n}, & \mathrm{if}~(\bba_1,\ldots,\bba_d)=\mathbf{0};\\
        & \mathrm{if}~\exists~\bby\in \Fq^n~s.t.~\mathrm{for~all}~i\in [d],\\
         q^{dn}, & \bba_i=c_i(\bby,1)~\mathrm{for~some}\\
        &(c_1,\ldots,c_d)\in\Fq^{d}\setminus \{\mathbf{0}\};\\
        0, & \mathrm{otherwise}. 
    \end{array}\right.
    \end{equation*}
    
\end{proposition}
\begin{proof}
    By Proposition \ref{prop_character_eigenvector}, $\lambda=\langle z,\chi_{\bba_1,\ldots,\bba_d} \rangle$. Then, 
    \begin{align}
        \bar{\lambda}& 
        =\sum_{(f_{\tilde{\bba}_1},\ldots,f_{\tilde{\bba}_d})\in G_{\bbb_1,\ldots,\bbb_d}}\overline{z\left((f_{\tilde{\bba}_1},\ldots,f_{\tilde{\bba}_d})\right)}\chi_{\bba_1,\ldots,\bba_d}((f_{\tilde{\bba}_1},\ldots,f_{\tilde{\bba}_d})) \nonumber\\
        &=\sum_{(\tilde{\bba}_1,\ldots,\tilde{\bba}_d)\in \Fq^{d(n+1)}}\sum_{\bbx\in \Fq^n s.t.\forall i\in [d]\atop f_{\tilde{\bba}_i}(\bbx)=0}\chi_{\bba_1,\ldots,\bba_d}((f_{\tilde{\bba}_1},\ldots,f_{\tilde{\bba}_d})) \nonumber\\
        &=\sum_{(\tilde{\bba}_1,\ldots,\tilde{\bba}_d)\in \Fq^{d(n+1)}}\sum_{\bbx\in \Fq^n s.t.\forall i\in [d]\atop f_{\tilde{\bba}_i}(\bbx)=0}e\left(\mathrm{Tr}\left(\sum_{i=1}^{d}\langle \bba_i,\tilde{\bba}_i \rangle\right)\right).\label{eq1_prop_eigenvalue_TT^*}
    \end{align}
    Recall that $\bbb_1,\ldots,\bbb_d$ are fixed vectors in $(\mathbb{Z}^{+})^n$ with $\mathrm{gcd}(b_{i,j},q-1)=1$ for all $1\leq i\leq d$ and $1\leq j\leq n$. Thus, for each $(y_1,\ldots,y_n)\in\Fq^n$ that satisfies 
    \begin{equation}\label{eq1.5_prop_eigenvalue_TT^*}
        \sum_{j=1}^{n}\tilde{a}_{i,j}y_{j}+\tilde{a}_{i,n+1}=0,~\forall~i\in [d],
    \end{equation}
    there is a unique $(x_1,\ldots,x_n)\in \Fq^n$ such that $x_j^{b_{i,j}}=y_j$ for all $1\leq j\leq n$. That is, there is a one-to-one correspondence between $\bbx\in \Fq^n$ satisfying $f_{\tilde{\bba}_i}(\bbx)=0$ for all $i\in [d]$ and  $\bby\in \Fq^n$ satisfying (\ref{eq1.5_prop_eigenvalue_TT^*}). Therefore, (\ref{eq1_prop_eigenvalue_TT^*}) becomes
    \begin{align}
        \bar{\lambda}& 
        =\sum_{(\tilde{\bba}_1,\ldots,\tilde{\bba}_d)\in \Fq^{d(n+1)}}\sum_{\bby\in \Fq^n s.t.\forall i\in [d]\atop \langle\tilde{\bba}_i, (\bby,1)\rangle=0}e\left(\mathrm{Tr}\left(\sum_{i=1}^{d}\langle \bba_i,\tilde{\bba}_i \rangle\right)\right) \nonumber\\
        &=\sum_{\bby\in \Fq^n}\sum_{(\tilde{\bba}_1,\ldots,\tilde{\bba}_d)\in \Fq^{d(n+1)} s.t.\forall i\in [d]\atop \langle\tilde{\bba}_i, (\bby,1)\rangle=0}e\left(\mathrm{Tr}\left(\sum_{i=1}^{d}\langle \bba_i,\tilde{\bba}_i \rangle\right)\right) \nonumber\\
        &=\sum_{\bby\in \Fq^n}\prod_{i=1}^{d}\left(\sum_{\tilde{\bba}_i\in \Fq^{n+1} \atop \langle\tilde{\bba}_i, (\bby,1)\rangle=0}e\left(\mathrm{Tr}\left(\langle \bba_i,\tilde{\bba}_i \rangle\right)\right)\right)
        ,\label{eq2_prop_eigenvalue_TT^*}
    \end{align}
    where the last equality follows by $e\left(\mathrm{Tr}\left(\sum_{i=1}^{d}\langle \bba_i,\tilde{\bba}_i \rangle\right)\right)=\prod_{i=1}^{d}e\left(\mathrm{Tr}\left(\langle \bba_i,\tilde{\bba}_i \rangle\right)\right)$.

    Next, for each $\bby=(y_1,\ldots,y_n)\in \Fq^n$ and $1\leq i\leq d$, we compute $$\Gamma_i:=\sum_{\tilde{\bba}_i\in \Fq^{n+1} \atop 
    \langle\tilde{\bba}_i, (\bby,1)\rangle=0} e\left(\mathrm{Tr}\left(\langle \bba_i,\tilde{\bba}_i \rangle\right)\right).$$ 
    
    If $\bba_i=\mathbf{0}$, $e(\mathrm{Tr}(\langle \bba_i,\bbu \rangle))=1$ for any $\bbu\in \Fq^{n+1}$. Note that there are exactly $q^n$ different $\tilde{\bba}_i\in \Fq^{n+1}$ satisfying $\sum_{j=1}^{n}\tilde{a}_{i,j}y_{j}+\tilde{a}_{i,n+1}=0$. This leads to $\Gamma_i=q^n$.
    
    If $\bba_i=c(y_1,\ldots,y_n,1)$ for some $c\in \Fq^*$, then $\langle \bba_i,\tilde{\bba}_i \rangle=0$ holds for each $\tilde{\bba}_i\in \Fq^{n+1}$ satisfying $\sum_{j=1}^{n}\tilde{a}_{i,j}y_{j}+\tilde{a}_{i,n+1}=0$. This also leads to $\Gamma_i=q^n$.
    
    When $\bba_i\neq c(y_1,\ldots,y_n,1)$ for any $c\in \Fq$, then we have $\bba_{i}|_{[n]}-a_{i,n+1}\bby\neq \mathbf{0}$. Thus,
    \begin{align*}
        \Gamma_i
        &=\sum_{\tilde{\bba}_i\in \Fq^{n+1} \atop \langle\tilde{\bba}_i, (\bby,1)\rangle=0}e\left(\mathrm{Tr}\left(\sum_{j=1}^{n}a_{i,j}\tilde{a}_{i,j}+a_{i,n+1}\tilde{a}_{i,n+1}\right)\right)\\
        &=\sum_{\tilde{\bba}_i|_{[n]}\in \Fq^{n}}e\left(\mathrm{Tr}\left(\sum_{j=1}^{n}a_{i,j}\tilde{a}_{i,j}-a_{i,n+1}\sum_{j=1}^{n}\tilde{a}_{i,j}y_{j}\right)\right)\\
        &=\sum_{\tilde{\bba}_i|_{[n]}\in \Fq^{n}}e\left(\mathrm{Tr}\left(\langle \tilde{\bba}_i|_{[n]}, \bba_{i}|_{[n]}-a_{i,n+1}\bby \rangle\right)\right)=0.
    \end{align*}

    Note that for each $(\bba_1,\ldots,\bba_d)\in \Fq^{d(n+1)}$, there is at most one $\bby\in \Fq^n$ satisfying $$\bba_i=c_i(y_1,\ldots,y_n,1),~\forall~i\in [d],$$ 
    for some $\bbc=(c_1,\ldots,c_d)\in \Fq^d\setminus\{\mathbf{0}\}$. Therefore, by (\ref{eq2_prop_eigenvalue_TT^*}), we can conclude that
    \begin{align*}
       \bar{\lambda}& 
        =\sum_{\bby\in \Fq^n}\prod_{i=1}^{d}\Gamma_i\\
        &=\left\{\begin{array}{cc}
        q^{(d+1)n}, & \mathrm{if}~(\bba_1,\ldots,\bba_d)=\mathbf{0};\\
        & \mathrm{if}~\exists~\bby\in \Fq^n~s.t.~\mathrm{for~all}~i\in [d],\\
         q^{dn}, & \bba_i=c_i(\bby,1)~\mathrm{for~some}\\
        &(c_1,\ldots,c_d)\in\Fq^{d}\setminus \{\mathbf{0}\};\\ 
        0, & \mathrm{otherwise}.
    \end{array}\right.
    \end{align*}
    This completes the proof.
\end{proof}

The following theorem follows easily from Proposition \ref{prop_eigenvalue_TT^*}.

\begin{theorem}\label{thm_eigenvalue_TT^*}
    The character of $G_{\bbb_1,\ldots,\bbb_d}$ are eigenvectors of the operator $\phi_z$ with the eigenvalues listed in Table \ref{Table_2}.
    \begin{table}[h]
    \caption{Eigenvalues and eigenvectors of $\phi_z$}\label{Table_2}
    \centering
    \begin{tabular}{|c|c|c|}
    \hline
    Character $\chi_{\bba_1,\ldots,\bba_d}$ & No. of  characters & Eigenvalue \\ 
    \hline
    $(\bba_1,\ldots,\bba_d)=\mathbf{0}$ & $1$ & $q^{(d+1)n}$ \\ \hline
    $\exists~\bby\in \Fq^n~s.t.~\mathrm{for~all}~i\in [d], \bba_i=c_i(\bby,1)$ & \multirow{2}{*}{$q^{n+d}-q^n$} & \multirow{2}{*}{$q^{dn}$} \\ 
    $\mathrm{for~some}~(c_1,\ldots,c_d)\in\Fq^{d}\setminus \{\mathbf{0}\}$ & & \\
    \hline
    else & $q^{d(n+1)}-q^{n+d}+q^n-1$ & $0$ \\ \hline
    \end{tabular}
\end{table}
\end{theorem}

\begin{corollary}\label{coro_left_singular_T}
    The normalized characters of $G_{\bbb_1,\ldots,\bbb_d}$ serve as left-singular vectors of the matrix $\bbT$, each associated with a singular value equals to the square root of its corresponding eigenvalue listed in Theorem \ref{thm_eigenvalue_TT^*}.
\end{corollary}

Building on Theorem \ref{thm_eigenvalue_TT^*}, we next establish a bound on the squared norm of the projection of an indicator vector onto the eigenspaces. 

\begin{lemma}\label{lem_proj_on_U}
    Let $\mathbf{1}_{\cV}\in \mathbb{C}^{q^{d(n+1)}}$ be the indicator vector of a set $\cV$ of varieties of the form (\ref{eq_variety_1}). Then, the squared norm of the projection of $\mathbf{1}_{\cV}$ onto the eigenspaces with eigenvalues $q^{(d+1)n}$ and $q^{dn}$ are $\frac{|\cV|^2}{q^{d(n+1)}}$ and at most
    \begin{equation*}
    \frac{|\cV|}{q^{(d-1)n}}\parenv{1-\frac{1}{q}+\frac{|\cV|(q^{d-1}-1)}{q^d}},
    \end{equation*}
    respectively.
\end{lemma}

\begin{proof}
    By the definition, every $V_{\bba_1,\ldots,\bba_d}\in \cV$ is uniquely determined by a $d$-tuple of vectors $(\bba_1,\ldots,\bba_d)$ with $\bba_i\in\Fq^{n+1}$ for all $1\leq i\leq d$. Thus, we can associate $\cV$ with a unique subset of vectors in $\Fq^{d(n+1)}$. With a slight abuse of notation, we also use $\cV$ to denote the associated subset of vectors in $\Fq^{d(n+1)}$. Consequently, in the following proof, we will also regard $\mathbf{1}_{\cV}$ as the indicator vector of the corresponding subset of vectors in $\Fq^{d(n+1)}$.

    The squared norm of the projection of $\mathbf{1}_{\cV}$ onto the normalized trivial character is
    $$\abs{\langle \mathbf{1}_{\cV},\frac{\chi_{\mathbf{0}}}{q^{d(n+1)/2}} \rangle}^2=\frac{|\cV|^2}{q^{d(n+1)}}.$$

    For $\bby\in \Fq^{n}$ and $\bbc\in \Fq^{d}\setminus\{\mathbf{0}\}$, denote $\chi_{(\bby,\bbc)}\triangleq \chi_{\bba_{1},\ldots,\bba_d}$ with $\bba_i=c_i(y_1,\ldots,y_n,1)$ for all $1\leq i\leq d$. By Theorem \ref{thm_eigenvalue_TT^*}, the eigenspace corresponding to eigenvalue $q^{dn}$ is spanned by the set of characters 
    $\{\chi_{(\bby,\bbc)}:~\bby\in \Fq^{n}~\mathrm{and}~\bbc\in \Fq^{d}\setminus\{\mathbf{0}\}\}$.
    
    Next, we bound the squared norm of the projection of $\mathbf{1}_{\cV}$ onto the eigenspace spanned by these characters: 
    \begin{align}
       & \sum_{\bby\in \Fq^{n} \atop \bbc\in \Fq^{d}\setminus\{\mathbf{0}\}}
       \abs{\langle \mathbf{1}_{\cV},\frac{\chi_{(\bby,\bbc)}}{q^{d(n+1)/2}}\rangle}^2 \nonumber\\ 
      = & q^{-d(n+1)}\sum_{\bby\in \Fq^{n} \atop \bbc\in \Fq^{d}\setminus\{\mathbf{0}\}}
      \sum_{\bbu\in \cV}\chi_{(\bby,\bbc)}(\bbu)\sum_{\tilde{\bbu}\in \cV}\overline{\chi_{(\bby,\bbc)}(\tilde{\bbu})} \nonumber\\
      = & q^{-d(n+1)}\sum_{\bby\in \Fq^{n} \atop \bbc\in \Fq^{d}\setminus\{\mathbf{0}\}}
      \sum_{\bbu,\tilde{\bbu}\in \cV}\chi_{(\bby,\bbc)}(\bbu-\tilde{\bbu}) \nonumber\\
      = & \frac{q^{n+d}-q^n}{q^{d(n+1)}}|\cV|+q^{-d(n+1)}\sum_{\bby\in \Fq^{n} \atop \bbc\in \Fq^{d}\setminus\{\mathbf{0}\}}
      \sum_{\bbu\neq\tilde{\bbu}\in \cV}\chi_{(\bby,\bbc)}(\bbu-\tilde{\bbu}).\label{eq1_proj_on_U}
    \end{align}
    
    For any $\bbu\neq \tilde{\bbu}\in \cV$, denote $\bbu=(\bbu_1,\ldots,\bbu_d)$ and $\tilde{\bbu}=(\tilde{\bbu}_1,\ldots,\tilde{\bbu}_d)$, where $\bbu_i=(u_{i,1},\ldots,u_{i,n+1})$ and $\tilde{\bbu}_i=(\tilde{u}_{i,1},\ldots,\tilde{u}_{i,n+1})$ are vectors in $\Fq^{n+1}$. Then, we have
    \begin{align}
        &\sum_{\bby\in \Fq^{n} \atop \bbc\in \Fq^{d}\setminus\{\mathbf{0}\}}
      \sum_{\bbu\neq\tilde{\bbu}\in \cV}\chi_{(\bby,\bbc)}(\bbu-\tilde{\bbu}) \nonumber\\
        = & \sum_{\bbu\neq\tilde{\bbu}\in \cV}\sum_{\bby\in \Fq^{n} \atop \bbc\in \Fq^{d}\setminus\{\mathbf{0}\}}e\left(\mathrm{Tr}\left(\sum_{i=1}^{d}\langle c_i(y_1,\ldots,y_n,1), \bbu_i-\tilde{\bbu}_i\rangle\right)\right) \nonumber\\
        = & \sum_{\bbu\neq\tilde{\bbu}\in \cV}\sum_{\bby\in \Fq^{n} \atop \bbc\in \Fq^{d}\setminus\{\mathbf{0}\}}e\left(\mathrm{Tr}\left(\sum_{i=1}^{d}\langle c_i\bby, (\bbu_i-\tilde{\bbu}_i)|_{[n]}\rangle\right)\right)e\left(\mathrm{Tr}\left(\sum_{i=1}^{d}c_i\left(u_{i,n+1}-\tilde{u}_{i,n+1}\right)\right)\right) \nonumber\\
        = & \sum_{\bbu\neq\tilde{\bbu}\in \cV}\sum_{\bbc\in \Fq^{d}\setminus\{\mathbf{0}\}}e\left(\mathrm{Tr}\left(\sum_{i=1}^{d}c_i\left(u_{i,n+1}-\tilde{u}_{i,n+1}\right)\right)\right)\left(\sum_{\bby\in \Fq^{n}}e\left(\mathrm{Tr}\left(\langle \bby, \sum_{i=1}^{d}c_i\left(\bbu_i-\tilde{\bbu}_i\right)|_{[n]}\rangle\right)\right)\right).\label{eq2_proj_on_U}
    \end{align}
    Notice that 
    \begin{equation*}
        \sum_{\bby\in \Fq^{n}}e\left(\mathrm{Tr}\left(\langle \bby, \sum_{i=1}^{d}c_i\left(\bbu_i-\tilde{\bbu}_i\right)|_{[n]}\rangle\right)\right)=\left\{\begin{array}{cc}
        0, & \mathrm{if}~\sum_{i=1}^{d}c_i(\bbu_i-\tilde{\bbu}_i)|_{[n]}\neq \mathbf{0};\\
        q^{n}, & \mathrm{otherwise}.
    \end{array}\right.
    \end{equation*}
    Thus, (\ref{eq2_proj_on_U}) becomes
    \begin{equation*}
      q^n\sum_{\bbu\neq\tilde{\bbu}\in \cV}\sum_{\bbc\in \Fq^{d}\setminus\{\mathbf{0}\}\atop \sum_{i=1}^{d}c_i(\bbu_i-\tilde{\bbu}_i)|_{[n]}= \mathbf{0}}e\left(\mathrm{Tr}\left(\sum_{i=1}^{d}c_i\left(u_{i,n+1}-\tilde{u}_{i,n+1}\right)\right)\right).
    \end{equation*}
    On the other hand, for each $\bbu\neq\tilde{\bbu}\in \cV$, denote $\bbw_{(\bbu,\tilde{\bbu})}\triangleq(u_{1,n+1}-\tilde{u}_{1,n+1},\ldots,u_{d,n+1}-\tilde{u}_{d,n+1})$ and    $$\bbM_{(\bbu,\tilde{\bbu})}\triangleq\left((\bbu_1-\tilde{\bbu}_1)'|_{[n]}~\cdots~(\bbu_d-\tilde{\bbu}_d)'|_{[n]}\right)$$
    as the $n\times d$ matrix with $(\bbu_i-\tilde{\bbu}_i)'|_{[n]}$ as its $i$-th column. Since the null space of the matrix $\bbM_{(\bbu,\tilde{\bbu})}$ is a subspace of $\Fq^d$ of size $q^{d-\mathrm{rank}(\bbM_{(\bbu,\tilde{\bbu})})}$, we have
    \begin{align*}
        \sum_{\bbc\in \Fq^{d}\setminus\{\mathbf{0}\}\atop \sum_{i=1}^{d}c_i(\bbu_i-\tilde{\bbu}_i)|_{[n]}= \mathbf{0}}e\left(\mathrm{Tr}\left(\sum_{i=1}^{d}c_i\left(u_{i,n+1}-\tilde{u}_{i,n+1}\right)\right)\right)&= \sum_{\bbc\in \Fq^{d}\setminus\{\mathbf{0}\}\atop \bbM_{(\bbu,\tilde{\bbu})}\cdot \bbc'=\mathbf{0}}e\left(\mathrm{Tr}\left(\langle \bbc,\bbw_{(\bbu,\tilde{\bbu})}\rangle\right)\right)\\
        &=\begin{cases}        -1,~\text{if}~\bbw_{(\bbu,\tilde{\bbu})}\neq\mathbf{0};\\
        q^{d-\mathrm{rank}(\bbM_{(\bbu,\tilde{\bbu})})}-1,~\text{otherwise}.
        \end{cases}
    \end{align*}
    Moreover, by $\bbu\neq\tilde{\bbu}$, we have $\mathrm{rank}(\bbM_{(\bbu,\tilde{\bbu})})\geq 1$ when $\bbw_{(\bbu,\tilde{\bbu})}=\mathbf{0}$. Thus, by the above discussion, we have
    \begin{align*}
        \sum_{\bby\in \Fq^{n} \atop \bbc\in \Fq^{d}\setminus\{\mathbf{0}\}}
        \sum_{\bbu\neq\tilde{\bbu}\in \cV}\chi_{(\bby,\bbc)}(\bbu-\tilde{\bbu})\leq q^{n}|\cV|(|\cV|-1)(q^{d-1}-1).
    \end{align*}
    Plugging this upper bound into (\ref{eq1_proj_on_U}), we have
    \begin{align*}
    \sum_{\bby\in \Fq^{n} \atop \bbc\in \Fq^{d}\setminus\{\mathbf{0}\}}
    \abs{\langle \mathbf{1}_{\cV},\frac{\chi_{(\bby,\bbc)}}{q^{d(n+1)/2}}\rangle}^2 &\leq 
    \frac{q^{n+d}|\cV|}{q^{d(n+1)}}\parenv{1-\frac{1}{q^d}+(|\cV|-1)\frac{q^{d-1}-1}{q^d}}\\
    &=\frac{|\cV|}{q^{(d-1)n}}\parenv{1-\frac{1}{q}+\frac{|\cV|}{q^d}(q^{d-1}-1)}.
    \end{align*}
    This completes the proof.
\end{proof}

\subsection{Proof of Theorem \ref{main_thm1}}

In this section, using the singular value decomposition of $\bbT$, along with Lemma \ref{lem_proj_on_V} and Lemma \ref{lem_proj_on_U}, we prove Theorem \ref{main_thm1}, which is restated below for the reader's convenience.

\mainthmone*

\begin{proof}[Proof of Theorem \ref{main_thm1}]
    Let $\mathbf{1}_{\cP}$ and $\mathbf{1}_{\cV}$ be the indicator vectors of the point set $\cP$ and the variety set $\cV$, respectively. Then, by Lemma \ref{lem_proj_on_V}, we can express $\mathbf{1}_{\cP}$ as
    \begin{equation}\label{eq_indicator_vec_P}
        \mathbf{1}_{\cP}=\frac{|\cP|}{q^{(n+d)/2}}\bbv_0+\alpha_1\bbv_1+\alpha_2\bbv_2,
    \end{equation}
    where $\bbv_0,\bbv_1$ and $\bbv_2$ are right-singular vectors of $\bbT$ that correspond to singular values $q^{(d+1)n/2}$, $q^{dn/2}$ and $0$, respectively, and $|\alpha_1|^2\leq (1-\frac{1}{q^d})|\cP|$. 
    
    Similarly, by Lemma \ref{lem_proj_on_U}, we can express $\mathbf{1}_{\cV}$ as
    \begin{equation}\label{eq_indicator_vec_V}
        \mathbf{1}_{\cV}=\frac{|\cV|}{q^{d(n+1)/2}}\bbu_0+\beta_1\bbu_1+\beta_2\bbu_2,
    \end{equation}
    where $\bbu_0,\bbu_1$ and $\bbu_2$ are left-singular vectors of $\bbT$ that correspond to singular values  $q^{(d+1)n/2}$, $q^{dn/2}$ and $0$, respectively, and 
    $$|\beta_1|^2\leq \frac{|\cV|}{q^{(d-1)n}}\parenv{1-\frac{1}{q}+\frac{|\cV|}{q^d}(q^{d-1}-1)}.
    $$ 

    Recall that $\bbT$ is the $q^{d(n+1)}\times q^{n+d}$ incidence matrix between all the varieties $V_{\bba_1,\ldots,\bba_d}$ defined by (\ref{eq_variety_1}) and all the points in $\Fq^{n+d}$. By the singular value decomposition of $\bbT$, we have $$\bbT\cdot \bbv_0'=q^{(d+1)n/2}\bbu_0',$$ 
    and $\bbT\cdot \bbv_1'=q^{dn/2}\tilde{\bbu}_1'$, where $\tilde{\bbu}_1$ is a unit norm left-singular vector that is orthogonal to $\bbu_0$ and $\bbu_2$. Then, $I(\cP,\cV)$, the number of incidences between varieties of the form (\ref{eq_variety_1}) and points in $\Fq^{n+d}$ satisfies
    \begin{align}
        I(\cP,\cV)=\mathbf{1}_{\cV}\cdot \bbT\cdot \mathbf{1}_{\cP}'&=\mathbf{1}_{\cV}\cdot \bbT\cdot \left(\frac{|\cP|}{q^{(n+d)/2}}\bbv_0+\alpha_1\bbv_1+\alpha_2\bbv_2\right)' \nonumber\\
        &=\mathbf{1}_{\cV}\cdot \left(\frac{|\cP|}{q^{(n+d)/2}}q^{(d+1)n/2}\bbu_0+q^{dn/2}\alpha_1\tilde{\bbu}_1\right)',\label{eq1_incidence}
    \end{align}
    Plugging (\ref{eq_indicator_vec_V}) into (\ref{eq1_incidence}), we have
    \begin{align*}
        I(\cP,\cV)
        &=\left(\frac{|\cV|}{q^{d(n+1)/2}}\bbu_0+\beta_1\bbu_1+\beta_2\bbu_2\right)\cdot \left(\frac{|\cP|}{q^{(n+d)/2}}q^{(d+1)n/2}\bbu_0+q^{dn/2}\alpha_1\tilde{\bbu}_1\right)' \\
        &=\frac{|\cP||\cV|}{q^d}+q^{dn/2}\alpha_1\beta_1\langle \bbu_1,\tilde{\bbu}_1 \rangle.
    \end{align*}
    Hence, 
    \begin{align}
         &\abs{I(\cP,\cV)-\frac{|\cP||\cV|}{q^d}}
        =\abs{q^{dn/2}\alpha_1\beta_1\langle \bbu_1,\tilde{\bbu}_1 \rangle} %\nonumber\\
        \leq q^{dn/2}\alpha_1\beta_1 \label{eq2_incidence}\\
        \leq & 
        q^{n/2}(1-\frac{1}{q^d})^{1/2}\sqrt{|\cP||\cV|}\parenv{1-\frac{1}{q}+\frac{|\cV|}{q^d}(q^{d-1}-1)}^{1/2} \label{eq3_incidence}.
    \end{align}
    where the inequality in (\ref{eq2_incidence}) follows from the fact that $\bbu_1,\tilde{\bbu}_1$ are unit vectors and (\ref{eq3_incidence}) follows from the upper bounds on $\alpha_1$ and $\beta_1$. When $d=1$, (\ref{eq3_incidence}) becomes $q^{n/2}(1-\frac{1}{q})\sqrt{|\cP||\cV|}$. When $d\geq 2$, (\ref{eq3_incidence}) is less than $$q^{n/2}\sqrt{|\cP||\cV|}\parenv{1+\frac{|\cV|}{q}}^{1/2}.$$ 
    This concludes the proof.
\end{proof}

\section{Applications}\label{sec: app}

\subsection{Point-flat incidences}

In this section, we present the proof of Theorem \ref{main_thm2} by applying Corollary \ref{coro_incidence_pts_flats} together with an averaging argument. For the reader's convenience, we restate the theorem below. 

\mainthmtwo*

Note that Corollary \ref{coro_incidence_pts_flats} only applies to a very specific family of $n$-flats as defined in \eqref{eq_flat_1}. Hence, one cannot apply it directly to prove Theorem \ref{main_thm2}. To circumvent this, we apply a simple averaging argument as follows.

Let $\mathcal{F}_0$ be the family of $n$-flats of the form \eqref{eq_flat_1}, and let $\mathcal{F}_1$ be the family of all $n$-flats in $\mathbb{F}_q^{n+d}$. Let $AGL(n+d,q)$ be the affine group that acts on $\mathbb{F}_q^{n+d}$. It is well known and easy to see that $AGL(n+d,q)$ acts transitively on $\cF_1$, and incidences are preserved under this action.

\begin{proof}[Proof of Theorem \ref{main_thm2}]
In the following, we will only prove the result for $d \geq 2$, as the proof for $d = 1$ follows similarly.

Let $g \in AGL(n+d,q)$ be an random element picked uniformly at random, and let $\mathcal{P}_g$ denote the image of the points in $\mathcal{P}$ under the action of $g$. Similarly, let $\mathcal{F}_g$ denote the image of the $n$-flats in $\mathcal{F}$ under the action of $g$ that are also in $\mathcal{F}_0$, i.e.,
\[ 
\mathcal{F}_g = \{g(F) : F \in \mathcal{F} \text{ and } g(F) \in \mathcal{F}_0\}.
\]

Then, by Corollary, \ref{coro_incidence_pts_flats}
\[
\left|I(\mathcal{P}_g, \mathcal{F}_g) - \frac{|\mathcal{P}_g||\mathcal{F}_g|}{q^d}\right| \leq q^{n/2} \sqrt{|\mathcal{P}_g||\mathcal{F}_g|} \left(1 + \frac{|\mathcal{F}_g|}{q}\right)^{1/2}.
\]
Note that $|\cP_g|=|\cP|$ holds for every $g\in AGL(n+d,q)$. Thus, taking the expectation on both sides of the above inequality, we get
\begin{equation}
\label{stamstam}
\left|\mathbb{E}\left[I(\mathcal{P}_g, \mathcal{F}_g)\right] - \frac{|\mathcal{P}|\mathbb{E}\left[|\mathcal{F}_g|\right]}{q^d}\right| \leq \mathbb{E}\left[q^{n/2} \sqrt{|\mathcal{P}_g||\mathcal{F}_g|} \left(1 + \frac{|\mathcal{F}_g|}{q}\right)^{1/2}\right].
\end{equation}

Next, if $p \in \mathcal{P}$ and $F \in \mathcal{F}$ form an incidence, i.e., $p \in F$, then their image under $g$ forms an incidence in $\mathcal{P}_g$ and $\mathcal{F}_g$ if and only if $g(F) \in \mathcal{F}_0$. Since the action of the affine group is transitive, or equivalently, by symmetry, $g(F) \in \mathcal{F}_0$ with probability $|\mathcal{F}_0| / |\mathcal{F}_1|$, hence 
\begin{equation}
\label{stam1}
\mathbb{E}\left[I(\mathcal{P}_g, \mathcal{F}_g)\right] = \frac{|\mathcal{F}_0|}{|\mathcal{F}_1|} I(\mathcal{P}, \mathcal{F}).
\end{equation}

By the same reasoning, we also have 
\begin{equation}
\label{stam2}
\mathbb{E}\left[|\mathcal{F}_g|\right] = \frac{|\mathcal{F}_0|}{|\mathcal{F}_1|} |\mathcal{F}|.
\end{equation}

Consider the function $f(x) = \sqrt{x + \frac{x^2}{q}}$ and note that it is concave for $x \geq 0$. Then, by Jensen's inequality and (\ref{stam2}), the RHS of \eqref{stamstam} is at most 
\begin{equation}
\label{stam3}
q^{n/2} \sqrt{|\mathcal{P}|} f\left(\frac{\sum_{g \in AGL(n+d,q)} |\mathcal{F}_g|}{|AGL(n+d,q)|}\right) = q^{n/2} \sqrt{|\mathcal{P}|} f\left(\frac{|\mathcal{F}_0||\mathcal{F}|}{|\mathcal{F}_1|}\right).
\end{equation}

By combining \eqref{stamstam}, \eqref{stam1}, \eqref{stam2}, and \eqref{stam3} and rearranging, we get that 
\begin{align}
\left|I(\mathcal{P}, \mathcal{F}) - \frac{|\mathcal{P}||\mathcal{F}|}{q^d}\right| 
& \leq \frac{|\mathcal{F}_1|}{|\mathcal{F}_0|} q^{n/2} \sqrt{|\mathcal{P}|} f\left(\frac{|\mathcal{F}_0||\mathcal{F}|}{|\mathcal{F}_1|}\right) \nonumber \\
& = \frac{|\mathcal{F}_1|}{|\mathcal{F}_0|} q^{n/2} \sqrt{|\mathcal{P}|} \sqrt{\frac{|\mathcal{F}_0||\mathcal{F}|}{|\mathcal{F}_1|} + \left(\frac{|\mathcal{F}_0||\mathcal{F}|}{|\mathcal{F}_1|}\right)^2 / q} \nonumber \\
& = q^{n/2} \sqrt{|\mathcal{P}|} \sqrt{\frac{|\mathcal{F}_1||\mathcal{F}|}{|\mathcal{F}_0|} + \frac{|\mathcal{F}|^2}{q}}. \label{stam4}
\end{align}

Now, it is left to determine the size of $\mathcal{F}_0$ and $\mathcal{F}_1$. Each $n$-flat of $\mathbb{F}_q^{n+d}$ is simply a coset of some $n$-dimensional subspace. Since there are $\qbinom{n+d}{n}_q$ $n$-dimensional subspaces in $\mathbb{F}_q^{n+d}$, and each has $q^d$ cosets, the size of $\mathcal{F}_1$ is
$q^d \qbinom{n+d}{n}_q$.

Lastly, we need to count the number of $n$-flats $F_{\mathbf{a}_1, \ldots, \mathbf{a}_d}$ of the form (\ref{eq_flat_1}). Since each such flat is uniquely defined by the $d$ vectors $\mathbf{a}_1, \ldots, \mathbf{a}_d$ with $\mathbf{a}_i \in \mathbb{F}_q^{n+1}$, there are $q^{d(n+1)}$ such flats.

Plugging this into \eqref{stam4}, we get
\[
\left|I(\mathcal{P}, \mathcal{F}) - \frac{|\mathcal{P}||\mathcal{F}|}{q^d}\right| \leq q^{n/2} \sqrt{|\mathcal{P}||\mathcal{F}|} \sqrt{\frac{\qbinom{n+d}{n}_q}{q^{dn}} + \frac{|\mathcal{F}|}{q}},
\]
as needed.
\end{proof}

\begin{remark}\label{rmk_rich_flats/points}
    Using their points-flats incidence bound, Lund and Saraf \cite{LS14} obtained several results analogous to Beck's Theorem \cite{Beck83} on the number of $t$-rich points and $t$-rich flats \footnote{Given a set of points $\cP$ and a set of $n$-flats $\cF$ in $\Fq^{n+d}$, a point is called $t$-rich if it's contained in at least $t$ different $n$-flats in $\cF$, and an $n$-flat is called $t$-rich if it contains at least $t$ different points in $\cF$}. These results state that one can find many $t$-rich points ($n$-flats) in the points-flats incidence graph as long as the number of $n$-flats (points) is large enough.

    One may wonder if we can use Theorem \ref{main_thm2} to improve Lund and Saraf's results on $t$-rich points and $t$-rich flats in \cite{LS14}. Unfortunately, since Theorem \ref{main_thm2} improves Theorem \ref{Lund's_thm} in \cite{LS14} only when the number of $n$-flats is relatively small, we cannot obtain better results using Theorem \ref{main_thm2}.
\end{remark}

\subsection{Pinned distances problem}

In this section, by a similar argument as in \cite{CILRR17} (see also \cite{PPSVV16} and \cite{PTV17}), we use Theorem \ref{main_thm1} to prove Corollary \ref{pinned_distance_improved}. For the reader's convenience, we restate the corollary below. 

\coropinneddistance*

\begin{proof}
    Fix a point $ \bbp = (p_1,\ldots,p_n) \in \cP $, recall that the variety defined by $ \bbp $ in Example \ref{ex2} is  
    $$V({\bbp}) \triangleq \left\{(x_1,\ldots,x_{n+1}) \in \mathbb{F}_q^{n+1} :~ x_{n+1} = \sum_{i=1}^n (x_i - p_i)^2 \right\}.$$
    Note that $\sum_{i=1}^n(x_i-p_i)^2=\sum_{i=1}^{n}x_i^2-\sum_{i=1}^{n}2p_i x_i+\sum_{i=1}^{n}p_i^2$. Thus, $V({\bbp})$ is the variety 
    $$V_{\bba}=V(x_{n+1}-(h(\bbx)+f_{\bba}(\bbx)))$$
    defined in (\ref{eq_variety_1}) with $h(\bbx)=\sum_{i=1}^{n}x_i^2$, $\bbb=(1,\ldots,1)$, $\bba=(-2p_1,\ldots,-2p_n,\sum_{i=1}^{n}p_i^2)$ and $f_{\bba}(\bbx)=-\sum_{i=1}^{n}2p_i x_i+\sum_{i=1}^{n}p_i^2$. Let $\cV=\{V({\bbp}):~\bbp\in \cP\}$. Since $V({\bbp})$ is the only variety in $\cV$ that contains the vector $(\bbp,0)$, the mapping from $\cP$ to $\cV$ that maps  $\bbp\in \cP$ to  $V({\bbp})$ is a bijection. Thus, we have $|\cV|=|\cP|$. 
    
    Define 
    $$\tilde{\cP}\triangleq\{(\bbp,t)\in \Fq^{n+1}:~\bbp\in \cP \text{ and } t\in \Delta(\cP,\bbp)\},$$
    and note that $|\tilde{\cP}|=\sum_{\bbp\in \cP}|\Delta(\cP,\bbp)|$. Then, it follows from Theorem \ref{main_thm1} for the case $d=1$ that
    \begin{equation}\label{eq_pin_distance_improved}
        \abs{I(\tilde{\cP},\cV)-\frac{|\tilde{\cP}||\cV|}{q}}\leq q^{n/2}(1-\frac{1}{q})\sqrt{|\tilde{\cP}||\cV|}.
    \end{equation}

    On the other hand, notice that a point $(\bbp,t)\in\tilde{\cP}$ is contained in $V(\tilde{\bbp})$ for some $\tilde{\bbp}\in \cP$ if and only if $t=\sum_{i=1}^{n}(p_i-\tilde{p_i})^2$. Hence, each ordered pair of points $(\bbp,\tilde{\bbp})\in \cP^2$ induces a unique incidence $\parenv{(\bbp,\sum_{i=1}^{n}(p_i-\tilde{p_i})^2),V(\tilde{\bbp})}$ in $I(\cP,\cV)$. This implies that $I(\tilde{\cP},\cV)=|\cP|^2$. Therefore, by (\ref{eq_pin_distance_improved}) and $|\cV|=|\cP|$, we have
    \begin{align}
        |\cP|^2&\leq \frac{|\tilde{\cP}||\cV|}{q}+q^{n/2}(1-\frac{1}{q})\sqrt{|\tilde{\cP}||\cV|}\nonumber\\
        &= \frac{|\cP|\sum_{\bbp\in \cP}|\Delta(\cP,\bbp)|}{q}+q^{n/2}(1-\frac{1}{q})\sqrt{|\cP|\sum_{\bbp\in \cP}|\Delta(\cP,\bbp)|}. \label{eq2_pin_distance_improved}
    \end{align}   
    If $\sum_{\bbp\in \cP}|\Delta(\cP,\bbp)|< (1-\varepsilon)q|\cP|$, then it follows from (\ref{eq2_pin_distance_improved}) that
    $$|\cP|^2< |\cP|^2(1-\varepsilon)+q^{\frac{n+1}{2}}(1-\frac{1}{q})|\cP|\sqrt{1-\varepsilon}.$$
    Equivalently,  $|\cP|< \frac{\sqrt{1-\varepsilon}}{\varepsilon}(q^{\frac{n+1}{2}}-q^{\frac{n-1}{2}})$, which   contradicts the assumption on the size of  $\cP$. Thus, $\frac{\sum_{\bbp\in\cP}|\Delta(\cP,\bbp)|}{|\cP|}\geq(1-\varepsilon)q$.

    Now, let $\cQ\triangleq \{\bbp\in \cP:~|\Delta(\cP,\bbp)|>(1-\sqrt{\varepsilon})q\}$. Suppose that $|\cQ|<(1-\sqrt{\varepsilon})|\cP|$, then we have
    \begin{align*}
        \sum_{\bbp\in \cP}|\Delta(\cP,\bbp)|&=\sum_{\bbp\in \cQ}|\Delta(\cP,\bbp)|+\sum_{\bbp\in \cP\setminus \cQ}|\Delta(\cP,\bbp)|\\
        &< q|\cQ|+(1-\sqrt{\varepsilon})q(|\cP|-|\cQ|)=(1-\sqrt{\varepsilon})q|\cP|+\sqrt{\varepsilon}q|\cQ|\\
        &<(1-\varepsilon)q|\cP|,
    \end{align*}
    which contradicts the former result that $\sum_{\bbp\in\cP}|\Delta(\cP,\bbp)|\geq(1-\varepsilon)q|\cP|$.
\end{proof}

\begin{remark}\label{rmk_pinned_distance}
    In \cite{PTV17}, Phuong et al. proved a bound on the number of incidences between points and generalized spheres. Using this bound, they generalized the result of Theorem \ref{pinned_distance_Chapman12} to pinned $Q$-distance problem, where $Q(\bbx)=\sum_{i=1}^{n}a_ix^{c_i}$ satisfies $2\leq c_i\leq N$ and $\gcd(c_i,q)=1$, and the set of $Q$-distances between the point $\bby$ and the set $\cP$ is defined as $\Delta_Q(\cP,\bby)\triangleq\left\{Q(\bbx-\bby):~\bbx\in \cP\right\}.$

    For $\cP\subseteq \Fq^{n}$ and $\bbp\in \cP$, one can define $V(\bbp)\triangleq\left\{(\bbx,x_{n+1})\in \Fq^{n+1}:~x_{n+1}=Q(\bbx-\bbp)\right\}$ and easily examine that $V(\bbp)$ is a special kind of variety defined in (\ref{eq_variety_1}). Then, by applying similar methods as those used in the proof of Corollary \ref{pinned_distance_improved}, we can also derive Phuong et al.'s result concerning the pinned $Q$-distance problem under a slightly weaker condition.
\end{remark}

\begin{remark}\label{rmk2_pinned_distance}
It is worth noting that the best-known results on the pinned distance problem for the case $d = 2$ over $\mathbb{F}_q$ are presented in \cite{MPPRS22} when $q$ is a prime, and in \cite{LSD16} for general $q$. These results were derived using a completely different approach, which relies on bounding the bisector energy quantity. 

The concept of bisector energy, first introduced by Lund, Sheffer, and De Zeeuw \cite{LSD16}, is defined as the number of pairs of segments with endpoints in the point set $ \cP $ that are symmetric with respect to a certain line, referred to as the bisector line.
\end{remark}

\section{Concluding remarks}\label{sec: conclusion}

In this paper, using spectral graph theory, we study an incidence problem for points and varieties of a certain form over finite fields. By a full characterization of the zero eigenspace of the considered points-varieties incidence graph, we show that the zero eigenspace of the incidence graph has a large dimension, which leads to improved incidence bounds when the size of the variety set is relatively small.

Note that the zero eigenvalue of any unbalanced bipartite graph $G=(A\cup B, E)$ ($|B|>|A|$) has multiplicity at least $|B|-|A|$. Thus, it's natural to wonder whether one can prove an ``expander crossing lemma'' for unbalanced bipartite graphs using the same approach as in the proof of Theorem \ref{main_thm1}. Recall that for the proof of Theorem \ref{main_thm1}, we need to determine the zero eigenspace of the graph. Since different graphs can have very different zero eigenspaces, one might not get good results for the general case. Nevertheless, it's still interesting if one can improve the result of Lemma \ref{expander_mix_bipartite} for certain classes of unbalanced bipartite graphs.

Recently, combinatorial tools are found to be powerful and useful for the study of incidence problems. One notable work is \cite{MST24} by Milojevi\'{c} et al., where the authors proved several incidence bounds for points and different geometric objects under non-degeneracy conditions using extremal graph theory. Their result matches the best known point-hyperplane incidence bounds in $\mathbb{R}^d$ by Apfelbaum and Sharir \cite{AS07} using geometrical space partitioning techniques. As another example, Iosevich et al. \cite{IPST24} improved the point-line incidence bound over finite fields by Vinh \cite{Vinh11} using tools from the VC-dimension theory. Hence, it would be also interesting to adopt these combinatorial tools to further improve results in this paper.

\section*{Acknowledgement}

We sincerely thank Prof.Jonathan M.Fraser and Dr.Firdavs Rakhmonov for kindly pointing out an error in the version of Theorem1.6 (namely, \cite[Theorem2.1]{BIP14}) that we cited in an earlier version of this paper, and for drawing our attention to the relevant works in \cite{Haemers80,BIP14_arxiv} related to Theorem~\ref{Lund's_thm}.

\bibliographystyle{plain}
\bibliography{biblio}

\begin{thebibliography}{10}

\bibitem{AC88}
Noga Alon and Fan~RK Chung.
\newblock Explicit construction of linear sized tolerant networks.
\newblock {\em Discrete Math.}, 72(1-3):15--19, 1988.

\bibitem{AS07}
Roel Apfelbaum and Micha Sharir.
\newblock Large complete bipartite subgraphs in incidence graphs of points and hyperplanes.
\newblock {\em SIAM J. Discrete Math.}, 21(3):707--725, 2007.

\bibitem{Beck83}
J{\'o}zsef Beck.
\newblock On the lattice property of the plane and some problems of {D}irac, {M}otzkin and {E}rd{\H{o}}s in combinatorial geometry.
\newblock {\em Combinatorica}, 3:281--297, 1983.

\bibitem{BIP14_arxiv}
Mike Bennett, Alex Iosevich, and Jonathan Pakianathan.
\newblock {Three-point configurations determined by subsets of $\mathbb{F}_q^2$ via the Elekes-Sharir Paradigm}.
\newblock {\em CoRR}, abs/1201.5039, 2012.

\bibitem{BIP14}
Mike Bennett, Alex Iosevich, and Jonathan Pakianathan.
\newblock {Three-point configurations determined by subsets of $\mathbb{F}_q^2$ via the Elekes-Sharir Paradigm}.
\newblock {\em Combinatorica}, 34(6):689--706, 2014.

\bibitem{BKT04}
Jean Bourgain, Nets Katz, and Terence Tao.
\newblock A sum-product estimate in finite fields, and applications.
\newblock {\em Geom. Funct. Anal.}, 14(1):27--57, 2004.

\bibitem{CEHIK12}
Jeremy Chapman, M~Burak Erdo{\u{g}}an, Derrick Hart, Alex Iosevich, and Doowon Koh.
\newblock Pinned distance sets, k-simplices, {W}olff's exponent in finite fields and sum-product estimates.
\newblock {\em Math. Z.}, 271(1):63--93, 2012.

\bibitem{CILRR17}
Javier Cilleruelo, Alex Iosevich, Ben Lund, Oliver Roche-Newton, and Misha Rudnev.
\newblock Elementary methods for incidence problems in finite fields.
\newblock {\em Acta Arith.}, 177(2):133--142, 2017.

\bibitem{Colbourn2010crc}
Charles~J Colbourn and Jeffrey~H Dinitz.
\newblock {\em The CRC handbook of combinatorial designs}.
\newblock CRC press, 2010.

\bibitem{DSV12}
Stefaan De~Winter, Jeroen Schillewaert, and Jacques Verstraete.
\newblock Large incidence-free sets in geometries.
\newblock {\em Electron. J. Combin.}, pages P24--P24, 2012.

\bibitem{Dvir12}
Zeev Dvir.
\newblock Incidence theorems and their applications.
\newblock {\em Foundations and Trends{\textregistered} in Theoretical Computer Science}, 6(4):257--393, 2012.

\bibitem{Grosu14}
Codru\c{t} Grosu.
\newblock {$\mathbb{F}_p$ is locally like $\mathbb{C}$}.
\newblock {\em J. Lond. Math. Soc. (2)}, 89(3):724--744, 03 2014.

\bibitem{Haemers80}
W.H. Haemers.
\newblock {\em {Eigenvalue Techniques in Design and Graph Theory}}.
\newblock Mathematical Centre tracts. Mathematisch Centrum, 1980.

\bibitem{Haemers95}
Willem~H Haemers.
\newblock Interlacing eigenvalues and graphs.
\newblock {\em Linear Algebra Appl.}, 226:593--616, 1995.

\bibitem{HR11}
Harald~Andr{\'e}s Helfgott and Misha Rudnev.
\newblock An explicit incidence theorem in $\mathbb{F}_p$.
\newblock {\em Mathematika}, 57(1):135--145, 2011.

\bibitem{IK09}
Alex Iosevich and Doowon Koh.
\newblock {The Erd{\"o}s--Falconer distance problem, exponential sums, and Fourier analytic approach to incidence theorems in vector spaces over finite fields}.
\newblock {\em SIAM J. Discrete Math.}, 23(1):123--135, 2009.

\bibitem{IPST24}
Alex Iosevich, Thang Pham, Steven Senger, and Michael Tait.
\newblock {Improved incidence bounds over arbitrary finite fields via the VC-dimension theory}.
\newblock {\em European J. Combin.}, 118:103928, 2024.

\bibitem{Jones12}
Timothy~GF Jones.
\newblock {Further improvements to incidence and Beck-type bounds over prime finite fields}.
\newblock {\em CoRR}, abs/1206.4517, 2012.

\bibitem{Kharazishvili24}
Alexander Kharazishvili.
\newblock {\em Introduction to Combinatorial Methods in Geometry}.
\newblock CRC Press, 2024.

\bibitem{KLP22}
Doowon Koh, Sujin Lee, and Thang Pham.
\newblock On the finite field cone restriction conjecture in four dimensions and applications in incidence geometry.
\newblock {\em Int. Math. Res. Not. IMRN}, 2022(21):17079--17111, 2022.

\bibitem{KPV21}
Doowon Koh, Thang Pham, and Le~Anh Vinh.
\newblock {Extension theorems and a connection to the Erd{\H{o}}s-Falconer distance problem over finite fields}.
\newblock {\em J. Funct. Anal.}, 281(8):109137, 2021.

\bibitem{Kollar15}
J\'{a}nos Koll\'{a}r.
\newblock {Szemer{\'e}di--Trotter--type theorems in dimension 3}.
\newblock {\em Adv. Math.}, 271:30--61, 2015.

\bibitem{Lewko15}
Mark Lewko.
\newblock New restriction estimates for the 3-d paraboloid over finite fields.
\newblock {\em Adv. Math.}, 270:457--479, 2015.

\bibitem{Lewko19}
Mark Lewko.
\newblock {Finite field restriction estimates based on Kakeya maximal operator estimates}.
\newblock {\em J. Eur. Math. Soc. (JEMS)}, 21(12):3649--3707, 2019.

\bibitem{lidl1997}
Rudolf Lidl and Harald Niederreiter.
\newblock {\em Finite fields}.
\newblock Cambridge university press, 1997.

\bibitem{LS14}
Ben Lund and Shubhangi Saraf.
\newblock Incidence bounds for block designs.
\newblock {\em SIAM J. Discrete Math.}, 30(4):1997--2010, 2016.

\bibitem{LSD16}
Ben Lund, Adam Sheffer, and Frank De~Zeeuw.
\newblock Bisector energy and few distinct distances.
\newblock {\em Discrete Comput. Geom.}, 56:337--356, 2016.

\bibitem{MST24}
Aleksa Milojevi{\'c}, Benny Sudakov, and Istv{\'a}n Tomon.
\newblock Incidence bounds via extremal graph theory.
\newblock {\em CoRR}, abs/2401.06670, 2024.

\bibitem{MPPRS22}
Brendan Murphy, Giorgis Petridis, Thang Pham, Misha Rudnev, and Sophie Stevens.
\newblock On the pinned distances problem in positive characteristic.
\newblock {\em J. Lond. Math. Soc. (2)}, 105(1):469--499, 2022.

\bibitem{PPSSVV18}
Thang Pham, Nguyen~Duy Phuong, Nguyen~Minh Sang, Claudiu Valculescu, and Le~Anh Vinh.
\newblock Distinct distances between points and lines in $\mathbb{F}_q^2$.
\newblock {\em Forum Math.}, 30(4):799--808, 2018.

\bibitem{PPSVV16}
N.D. Phuong, T.~Pham, N.M. Sang, C.~Valculescu, and L.A. Vinh.
\newblock Incidence bounds and applications over finite fields.
\newblock {\em CoRR}, abs/1601.00290, 2016.

\bibitem{PTV17}
Nguyen~D Phuong, Pham Thang, and Le~A Vinh.
\newblock Incidences between points and generalized spheres over finite fields and related problems.
\newblock {\em Forum Math.}, 29(2):449--456, 2017.

\bibitem{Sheffer22}
Adam Sheffer.
\newblock {\em Polynomial Methods and Incidence Theory}.
\newblock Cambridge Studies in Advanced Mathematics. Cambridge University Press, 2022.

\bibitem{ST83}
Endre Szemer{\'e}di and William~T. Trotter.
\newblock Extremal problems in discrete geometry.
\newblock {\em Combinatorica}, 3:381--392, 1983.

\bibitem{Zachi23}
Itzhak Tamo.
\newblock {Points-Polynomials Incidence Theorem with an Application to Reed-Solomon Codes}.
\newblock {\em CoRR}, abs/2312.12962, 2023.

\bibitem{TV06}
Terence Tao and Van~H Vu.
\newblock {\em Additive combinatorics}, volume 105.
\newblock Cambridge University Press, 2006.

\bibitem{Vinh11}
Le~Anh Vinh.
\newblock {The Szemer{\'e}di--Trotter type theorem and the sum-product estimate in finite fields}.
\newblock {\em European J. Combin.}, 32(8):1177--1181, 2011.

\end{thebibliography}

\end{document}